\documentclass[a4paper,11pt,reqno]{amsart}
\usepackage{amsmath,amssymb,amsthm} 
\usepackage{graphics} 

\usepackage[hyperref,doi,url=false,isbn=false,  backref,style=ieee,maxbibnames=99, backend=biber]{biblatex}
\bibliography{MST-backup.bib}

\usepackage{pgf}
\usepackage{tikz,wrapfig}
\usetikzlibrary{arrows}

\usepackage[colorlinks = true]{hyperref}
\usepackage{epsfig}
\usepackage{color}

\numberwithin{equation}{section}

\newtheorem{theorem}{Theorem}[section]
\newtheorem{lemma}[theorem]{Lemma}
\newtheorem{proposition}[theorem]{Proposition}

\theoremstyle{definition}
\newtheorem{definition}[theorem]{Definition}
\newtheorem{remark}[theorem]{Remark}

\newcommand{\R}{\mathbb{R}}

\newcommand{\C}{\mathbf{C}}

\newcommand{\be}{\begin{equation}}
\newcommand{\ee}{\end{equation}}

\newcommand{\dist}{\operatorname{dist}}
\newcommand{\n}{\mathsf{n}}

\def \C{\mathcal{C}_{MST}}

\def\EE{\mathbb{E}}

\newcommand{\bra}[1]{\left( #1 \right)}
\newcommand{\sqa}[1]{\left[ #1 \right]}
\newcommand{\cur}[1]{\left\{ #1 \right\}}

\newcommand{\abs}[1]{\left| #1 \right|}

\usepackage{enumerate}

\newcommand{\T}{\mathbb{T}^d}

\title{On Minimum Spanning Trees for Random Euclidean Bipartite Graphs}
\thanks{D.T.\ was partially supported by Gnampa project 2020 ``Problemi di ottimizzazione con vincoli via trasporto ottimo e incertezza''.}
\author[M. Correddu]{Mario Correddu}
\address{M.C.: Dipartimento di Matematica, Università degli Studi di Pisa, 56125 Pisa, Italy  }
\email{m.correddu@studenti.unipi.it}
\author[D. Trevisan]{Dario Trevisan}
\address{D.T.: Dipartimento di Matematica, Università degli Studi di Pisa, 56125 Pisa, Italy  }
\email{dario.trevisan@unipi.it}
\date{\today}
\keywords{Euclidean functionals, minimum spanning tree, geometric probability}

\begin{document}

\begin{abstract}
We consider the minimum spanning tree problem on a weighted complete bipartite graph $K_{n_R, n_B}$ whose $n=n_R+n_B$ vertices are random, i.i.d.~uniformly distributed points in the unit cube in $d$ dimensions and edge weights are the $p$-th power of their Euclidean distance, with $p>0$. In the large $n$ limit with $n_R/n \to \alpha_R$ and  $0<\alpha_R<1$, we show that the maximum vertex degree of the tree grows logarithmically, in contrast with the classical, non-bipartite, case, where a uniform bound holds depending on $d$ only. Despite this difference, for $p<d$, we are able to prove that the total edge costs normalized by the rate $n^{1-p/d}$ converge to a limiting constant that can be represented as a series of integrals, thus extending a classical result of Avram and Bertsimas to the bipartite case and confirming a conjecture of Riva, Caracciolo and Malatesta.
\end{abstract}
\maketitle

\section{Introduction}

The minimum spanning tree (MST) problem ranks among the simplest Combinatorial Optimization problems, with many applications, well beyond its historical introduction for network design \cite{graham_history_1985} including approximation algorithms for more complex problems \cite{christofides1976worst, kou_fast_1981} and cluster analysis \cite{asano_clustering_1988}. %

Its formulation is straightforward: given a weighted undirected graph $G=(V,E,w)$ with $w : E \to (0, \infty)$, find a sub-graph $T\subseteq E$ that connects all nodes $V$ and has a minimal total edge cost
\[ \sum_{e \in T} w(e),\]
thus defining the MST cost functional $\C(G)$. 
Minimality yields that redundant connections can be discarded, so that the resulting sub-graph $T$ turns out to be a tree, i.e., connected and without cycles. Several algorithms have been proposed for its solution, from classical greedy to more efficient ones \cite{chazelle2000minimum}, possibly randomized \cite{karger_randomized_1995}.

Despite its apparent simplicity, a probabilistic analysis of the problem, i.e., assuming that weights are random variables with a given joint law and studying the resulting random costs and MST's yields interesting results. Moreover, it may suggest mathematical tools to deal with more complex problems, such as the Steiner tree problem or the travelling salesperson problem, where one searches instead for a cycle connecting all points having minimum total edge weight.

The most investigated random model is surely that of i.i.d.~weights with a regular density, as first studied by Frieze \cite{frieze_value_1985}, who showed in particular the following law of large numbers: if $G^n= (V^n, E^n, w^n)$, with $(V^n,E^n)= K_n$ the complete graph over $n$ nodes and $w^n = \bra{w_{ij}}_{i,j =1}^n$ are independent and uniformly distributed on $[0,1]$, then almost surely
\begin{equation}\label{eq:frieze} \lim_{n\to \infty}   \C(G^n) = \zeta(3) = \sum_{k=1}^\infty \frac{1}{k^3}.\end{equation}

Another well studied setting is provided by Euclidean models, where nodes are i.i.d.\ sampled points in a region (say uniformly on a cube $[0,1]^d \subseteq \R^d$, for simplicity) and edge weights are functions of their distance, e.g. $w(x,y) = |x-y|^p$ for some parameter $p>0$. This setting dates back at least to the seminal paper by Beardwood, Halton and Hammersley \cite{beardwood_shortest_1959} where they focused on the travelling salesperson problem, but stated that other problems may be as well considered, including the MST one. A full analysis was later performed by Steele \cite{steele_growth_1988} who proved that, if the Euclidean graph consists of $n$ nodes, then for every $0<p<d$, almost sure convergence holds 
\begin{equation}\label{eq:steele} \lim_{n\to \infty} \frac{ \C(G^n)}{n^{1-p/d}} =  \beta_{MST}(p,d),\end{equation}
where $\beta_{MST}(p,d) \in (0,\infty)$ is a constant. The rate $n^{1-p/d}$ is intuitively clear due to the fact that there are $n-1$ edges a tree over $n$ points and the typical distance between two adjacent points is expected to be of order $n^{-1/d}$. The constraint $p<d$ was removed by Aldous and Steele \cite{aldous_asymptotics_1992} and Yukich  \cite{yukich_asymptotics_2000}, so that convergence holds in fact for any $p>0$. This result can be seen as an application of a general Euclidean additive functional theory \cite{steele_subadditive_1981, yukich_probability_1998}. However, such general methods that work for other combinatorial optimization problems give not much insight on the precise value of the limit constant $\beta_{MST}(p,d)$. The MST problem is known to be exceptional, for a (sort of) explicit series representation, analogue  to \eqref{eq:frieze}, was obtained by Avram and Bertismas \cite{avram_minimum_1992}, although only in the range $0<p<d$. The latter was used by Penrose \cite{penrose_random_1996}, in connection with continuum percolation, to study, among other things, the MST in the high dimensional regime $d\to \infty$. An alternative approach towards explicit formulas was proposed by Steele \cite{steele_minimal_2002}, but limited to the case of i.i.d.\ weights, based on Tutte polynomials.

\vspace{1em}

 Aim of this paper is to investigate analogous results for \emph{bipartite} Euclidean random models, i.e., when nodes correspond to two distinct families of sampled points (e.g., visually rendered by red/blue colourings) and weights, still  given by a power of the distance, are only defined between points with different colours. Formally, we replace the underlying complete graph $K_n$ with a complete bipartite graph $K_{n_R, n_B}$ with $n_R+n_B=n$. 

A similar question was formulated and essentially solved in the model with independent weights by Frieze and McDiarmid \cite{frieze_random_1989}. In Euclidean models, however, it is known that such innocent looking variant may in fact cause quantitative differences in the corresponding asymptotic results. For example, in the Euclidean bipartite travelling salesperson problem with $d=1$ and $d=2$, the correct asymptotic rates (for $p=1$) are known to be respectively $\sqrt{n}$ \cite{caracciolo_solution_2018} and $\sqrt{n \log n }$ \cite{capelli_exact_2018}, larger than the natural $n^{1-1/d}$ for the non-bipartite problem.  Similar results are known for other problems, such as the minimum matching problem \cite{steele_subadditive_1981} and its bipartite counterpart, also related to the optimal transport problem \cite{ajtai_optimal_1984, talagrand_ajtai-komlos-tusnady_1992, talagrand_upper_2014, ambrosio_pde_2019, caracciolo_scaling_2014}. Barthe and Bordenave proposed a bipartite extension of the Euclidean additive functional theory \cite{barthe_combinatorial_2013} that allows to recover an analogue of \eqref{eq:steele} for many relevant combinatorial optimization problems on bipartite Euclidean random models, although its range of applicability is restricted to $0<p<d/2$ (the cases $p=1$, $d\in\cur{1,2}$ are indeed outside this range) and anyway the MST problem does not fit in the theory. The main reason for the latter limitation is that there is no uniform bound on the maximum degree of a MST on a bipartite Euclidean random graph -- their theory instead applies to a variant of the problem where a uniform bound on the maximum degree is imposed, which is in fact algorithmically more complex (if the bound is two it recovers essentially the travelling salesperson problem). 
 
 \subsection*{Main results}
 Our first main result describes precisely the asymptotic maximum degree of a MST on a bipartite Euclidean random graph, showing that it grows logarithmically in the total number of nodes, in the asymptotic regime where a fraction of points $\alpha_R \in (0,1)$ is red and the remaining $\alpha_B = 1-\alpha_R$ is blue.

\begin{theorem}\label{thm:main-degree}
Let $d\ge 1$, let  $n \ge 1$ and  $R^n = \bra{X_i}_{i=1}^{n_R}$, $B^n= \bra{Y_i}_{i=1}^{n_B}$ be (jointly) i.i.d.~uniformly distributed on $[0,1]^d$ with $n_R+n_B = n$ and
\[ \lim_{n \to \infty} \frac{n_R}{n} = \alpha_R \in (0,1), \quad \lim_{n \to \infty} \frac{n_B}{n} = \alpha_B = 1 -\alpha_R.\]
Let $T^n$ denote the MST over the complete bipartite graph with independent sets $R^n$, $B^n$ and weights $w(X_i,Y_j) = |X_i-Y_j|$, and let $\Delta(T^n)$ denote its maximum vertex degree. Then, there exists a constant $C= C(d,\alpha_R)>0$ such that
\[ \lim_{ n \to \infty} P\bra{ C^{-1} < \frac{\Delta(T^n)}{\log(n)} < C} = 1.\]
\end{theorem}

(Indeed, the structure of the MST does not depend on the specific choice of the exponent $p>0$, so we simply let $p=1$ above). The proof is detailed in Section~\ref{sec:main-degree}.

Our second main result shows that, although the general theory of Barthe and Bordenave does not apply and the maximum degree indeed grows, the total weight cost for the bipartite Euclidean MST problem turns out to be much closer to the non-bipartite one, since no exceptional rates appear  in low dimensions. Before we give the complete statement, let us introduce the following quantity, for $d \ge 1$, $k_R$, $k_B \ge 1$, 
$\alpha_R \in (0,1)$,
\[ \begin{split} E(k_R, k_B, \alpha_R) = & \int_{\Theta(k_R, k_B)} \bra{ \alpha_R |D(\cur{b_j}_{j=1}^{k_B})| + \alpha_B | D(\cur{r_i}_{i=1}^{k_R})|}^{-(k_R+k_B)/d}\cdot  \\
& \quad \cdot \bra{  \frac{k_R}{\alpha_R} \delta_0(r_1)d b_1 +  \frac{k_B}{\alpha_B}d r_1 \delta_0(b_1)} dr_2 \ldots d r_{k_R} db_2 \ldots d b_{k_B},\end{split} \]
where  $\alpha_B = 1-\alpha_R$ and we write 
\begin{equation}\label{eq:def-theta} \Theta(k_R, k_B) \subseteq (\R^d)^{k_R} \times (\R^d)^{k_B},\end{equation}
 for the set of (ordered) points $( \bra{r_i}_{i=1}^{k_R} , \bra{b_j}_{j=1}^{k_B})$ such that, in the associated Euclidean bipartite graph with weights  $\bra{ |r_i-b_j|}_{i,j}$, the subgraph with all edges having weight less than $1$ is connected (or equivalently, there exists a bipartite Euclidean spanning tree having all edges with length weight less than $1$), and for a set $A \subseteq \R^d$, we write
 \begin{equation}
 \label{eq:def-da}  D(A) = \cur{x \in \R^d\, : \, \dist(A,x)  \le 1},\end{equation}
 and $|D(A)|$ for its Lebesgue measure. Notice also that the overall integration is performed with respect to Lebesgue measure over $k_R+k_B-1$ variables in $\R^d$ and one (either $r_1$ or $b_1$) is instead with respect to a Dirac measure at $0$.

These quantities enter in the explicit formula for the limit constant in the bipartite analogue of \eqref{eq:steele}, as our second main result shows.

\begin{theorem}\label{thm:main-bipartite-euclidean}
Let $d\ge 1$, let  $n \ge 1$ and  $R^n = \bra{X_i}_{i=1}^{n_R}$, $B^n= \bra{Y_i}_{i=1}^{n_B}$ be (jointly) i.i.d.~uniformly distributed on $[0,1]^d$ with $n_R+n_B = n$ and
\[ \lim_{n \to \infty} \frac{n_R}{n} = \alpha_R \in (0,1), \quad \lim_{n \to \infty} \frac{n_B}{n} = \alpha_B = 1 -\alpha_R.\]
Let $T^n$ denote the MST over the complete bipartite graph with independent sets $R^n$, $B^n$ and weights $w(X_i,Y_j) = |X_i-Y_j|$. Then, for every $p \in (0, d)$, the following convergence holds
\begin{equation}\label{eq:lln-bip} \lim_{n \to \infty} \frac{ \EE\sqa{ \sum_{\cur{X_i,Y_j} \in T^n} |X_i - Y_j|^p}}{n^{1-p/d}} = \beta_{bMST}(d,p),\end{equation}
and  the constant $\beta_{bMST}(d,p) \in (0, \infty)$ is given by  the series
\begin{equation}\label{eq:bip-ab} \beta_{bMST}(d,p) = \frac{p}{d}\sum_{k_R,k_B=1}^{\infty} \frac{\alpha_R^{k_R}}{k_R! }\frac{ \alpha_B^{k_B}}{k_B!} \frac{\Gamma((k_R+k_B)/d)}{k_B+k_R} E(k_R, k_B, \alpha_R).\end{equation}
Moreover, if $p<d/2$ for $d \in \cur{1,2}$ or $p<d$ for  $d \ge 3$, convergence is almost sure:
\[ \lim_{n \to \infty} \frac{  \sum_{\cur{X_i,Y_j} \in T^n} |X_i - Y_j|^p}{n^{1-p/d}} = \beta_{bMST}(d,p).\]
\end{theorem}

The proof is detailed in Section~\ref{sec:avram-berstimas}. The one-dimensional random bipartite Euclidean MST has been recently theoretically investigated in the statistical physics literature by Riva, Caracciolo and Malatesta \cite{caracciolorandom}, together with extensive numerical simulations also in higher dimensions, hinting at the possibility of a non-exceptional rate $n^{1-p/d}$ also for $d=2$. In particular, our result confirms this asymptotic rate in the two dimensional case, with a.s.\ convergence if $p<1$ and just convergence of the expected costs if $1\le p <2$ -- in fact we also have a general upper bound if $p\ge 2$ (Lemma~\ref{lem:boundexp-bip-any-p}).

\subsection*{Further questions and conjectures}

Several extensions of the results contained in this work may be devised, for example by generalizing to $k$-partite models or more general block models, allowing for weights between the same coloured points but possibly with a different function, e.g.\ the same power of the distance function, but multiplied by a different pre-factor according to pair of blocks. An interesting question, also open for the non-bipartite case, is to extend the series representation for the limiting constant to the case $p\ge d$. On the other side, we suspect that additivity techniques may yield convergence in \eqref{eq:lln-bip} also in the range $p \ge d$, without an explicit series, but we leave it for future explorations. A further question, that has no counterpart in the non-bipartite case, is what happens if the laws of different coloured points are different, say with densities $f_R$ and $f_B$ that are regular, uniformly positive and bounded. Assuming that $n_R/n \to 1/2$, a natural conjecture is that the limit holds with \eqref{eq:bip-ab} obtained by replacing $\alpha_R$ and $\alpha_B$  with the ``local'' fraction of points $f_R(x)/2$, $f_R(x)/2$ and then integrating with respect to $x \in [0,1]^d$, i.e.,
\begin{equation}\label{eq:conjecture-ab}  \frac{p}{d}\sum_{k_R,k_B=1}^{\infty}\frac{1}{2^{k_R+k_B}}\int_{[0,1]^d} \frac{f_R^{k_R}}{k_R! }\frac{ f_B^{k_B}}{k_B!} \frac{\Gamma((k_R+k_B)/d)}{k_B+k_R} E(k_R, k_B, f_R/2, f_B/2).\end{equation}
Finally, a central limit theorem is known for MST problem \cite{kesten_central_1996, chatterjee_minimal_2017} and it may be interesting to understand the possible role played by the additional fluctuations introduced in bipartite setting for analogue results.

\subsection*{Structure of the paper}
In Section~\ref{sec:notation} we collect useful notation and properties of general MST's, together with crucial observations in the metric setting (including the Euclidean one) and some useful probabilistic estimates. We try here to keep separate as much as possible probabilistic from deterministic results,  to simplify the exposition. In Section~\ref{sec:main-degree} we prove Theorem~\ref{thm:main-degree} and in Section~\ref{sec:avram-berstimas} we first extend  \cite[Theorem 1]{avram_minimum_1992} to the bipartite case and then apply it in the Euclidean setting. An intermediate step requires to argue on the flat torus $\T$ to exploit further homogeneity. We finally use a concentration result to obtain almost sure convergence: since the vertex degree is not uniformly bounded, the standard inequalities were not sufficient to directly cover the case $p<1$, so we prove a simple variant of McDiarmid inequality in Appendix~\ref{app:mcdiarmid} that we did not find in the literature and may be of independent interest.

\section{Notation and preliminary results}\label{sec:notation}

\subsection{Minimum spanning trees}

Although our focus is on weighted graphs induced by points in the Euclidean space $\R^d$, the following general definition of minimum spanning trees will be useful.

\begin{definition}\label{def:mst-general}
Given a weighted undirected finite graph $G = (V,E, w)$, with $w: E \to [0, \infty]$, the MST cost functional is defined as
\begin{equation}\label{eq:mst-cost} \C(G) = \inf \cur{ \sum_{e \in T} w(e)\, : \, \text{$T \subseteq E$ is a connected spanning subgraph}},\end{equation}
\end{definition}

Here and below, connected is in the sense that only edges with finite weight must be considered. We consider only minimizers $T$ in \eqref{eq:mst-cost} that are trees, i.e., connected and acyclic, otherwise removing the most expensive edge in a cycle would give a competitor with smaller cost (since we assume possibly null weights, there may be other minimizers). The following lemma is a special case of the cut property of minimum spanning trees, but will play a crucial role in several occasions, so we state it here. 

\begin{lemma}\label{lem:voronoi}
Let $G = (V, E, w)$, $v \in V$ and assume that $e \in \operatorname{arg min}\cur{w(f)\, : \, v \in f}$ is unique. Then, $e$ belongs to every minimum spanning tree of $G$.  
\begin{proof}
 Assume that $e$ does not belong to a minimum spanning tree $T$. Addition of $e$ to $T$ induces a cycle that includes  necessarily another edge $f$, with $v \in f$ and by assumption $w(f)>w(e)$. By removing $f$, the cost of the resulting connected graph is strictly smaller that the cost of $T$, a contradiction.
\end{proof}
\end{lemma}

We write $\n_G(v) \in V$, or simply $\n(v)$ if there are no ambiguities, for the closest node to $v$ in $G$, i.e.,
\[ e = \cur{v, \n(v)} \in \operatorname{arg min}_{w \in E}\cur{w(f)\, : \, v \in f},\]
assuming that such node is unique.

The subgraphs
\[ G(z) := \cur{ e \in E \, : \, w(e) \le z}, \quad \text{for $z \ge 0$,}\]
are strongly related to the minimum spanning tree on $G$, since the execution of Kruskal's algorithm yields the identity, already observed in \cite{avram_minimum_1992},
\begin{equation}\label{eq:ab-general}
 \C(G) =\int_{0}^{\infty}\bra{ C_{G(z)}-1}dz,
\end{equation}
where we write $C_G$ denotes the number of connected components of a graph $G$. Indeed, the function $z \mapsto C_{G(z)}$ is piecewise constant and decreasing from $|V|$ towards $1$ (assuming that all weights are strictly positive and $G$ is connected). Assume for simplicity that all weights $\cur{w_e}_{e \in E}$ are different, so that $z \mapsto C_{G(z)}$ has only unit jumps, on a set $J_{-}$. An integration by parts gives the identity 
\[
\int_{0}^{\infty}\bra{C_{G(z)}-1}dz= 
 \sum_{z\in J_{-}} z. \]
To argue that the right hand side is the cost of a MST, e.g., obtained by Kruskal's algorithm, we may represent the connected components of $G(z)$ as a function of $z$ in a tree-like graph (see Fig.~\ref{fig:prim}): starting with components consisting of single nodes at $z=0$, whenever two components merge (i.e., at values $z \in J_{-}$) we connect the corresponding segments. This yields a (continuous) tree with leaves given by the nodes and a root at $z =\infty$. Since Kruskal's algorithm returns exactly the tree consisting of the edges corresponding to such $z \in J_{-}$, we obtain \eqref{eq:ab-general}. 

\begin{remark}\label{rem:p-dependence-mst}
The construction above also yields that the minimum spanning trees of $G = (V, E, w)$ are also minimum spanning trees associated to the graph $G^\psi = (V,E, \psi\circ w)$, i.e., weights are $\psi(w(e))$ where $\psi$ is an increasing function.  In particular, assuming that $\psi:[0, \infty)$ is strictly increasing with $\psi(0) = 0$, then $G^\psi(z) = G(\psi^{-1}(z))$, hence
\[ \C(G^\psi) = \int_0^\infty \bra{ C_{G^\psi(z)} - 1 } d z = \int_0^\infty \bra{C_{G(u)} -1 }d \psi(u).\]
In particular, choosing $\psi(x) = x^p$ and letting $p \to \infty$, we obtain that any minimum spanning tree $T$ is also a minimum bottleneck spanning tree, i.e., $T$ minimizes the functional
\begin{equation}\label{eq:min-bottleneck}\C^\infty(G) := \inf \cur{ \max_{e \in T}w(e)\, : \, \text{$T \subseteq E$ is a connected spanning subgraph}}.\end{equation}
\end{remark}

A similar argument \cite[Lemma 4]{avram_minimum_1992} yields an upper bound for a similar quantity where $C_{G}$ is replaced with $C_{k,G}$, the number of connected components having at least $k$ nodes.

\begin{lemma}\label{lem:prim}
Let $G = (V, E, w)$ be connected with all distinct weights $(w(e))_{e \in E}$ (if finite) and
 $2 \le k \le |V|$. 
Then, there exists a partition $V = \bigcup_{i=1}^{m} C_i$ such that letting $G_k$ be the graph over the node set $\cur{C_1,\ldots, C_m}$ with weights \[w(C_i,C_j) = \inf \cur{w(e) \, : \, e = \cur{x,y} \in V, x\in C_i, y \in C_j},  \quad \text{for $i$, $j \in \cur{1, \ldots, m}$,}\]
then
\begin{equation} \label{eq:ab-bound-ckg} \int_0^\infty ( C_{k,G(z)} -1 ) d z  \le \C(G_k).\end{equation}
Moreover, for every $i=1, \ldots, m$, $|C_i| \ge k$, hence $m \le |V|/k$, and there exists $v \in C_i$ such that $\n_G(v) \in C_i$ and $\n_G(\n_G(v)) = v$. 
\begin{proof}


The function $z \mapsto C_{k, G(z)}$ is piecewise constant, with jumps of absolute size $1$, with positive sign on a set $J_+$ and negative sign on a set $J_-$. 
An integration by parts gives
\[ \int_{0}^{\infty}\bra{ C_{k,G(z)}(z)-1}dz=\sum_{z \in J_{-}}z-\sum_{z\in J_{+}}z\leq
\sum_{z \in J_{-}}z.\]
We interpret the right hand side above as $\C(G_k)$ for a suitable graph $G_k$. To define the sets $C_i$, we let $J_+ = \cur{z_1, \ldots, z_{m}}$ and define, for every $z_i$, the ``seed'' of $C_i$  as the set of nodes that gives an additional component with at least $k$ nodes, i.e., obtained by merging two components in $G(z_i^-)$, both having less that $k$ nodes. Notice that, since $C_i$ will be then completed by adding nodes to such seeds, the last statement is already fulfilled. Indeed, any seed contains at least $k$ nodes we can  always choose $v_1$, $v_2$ in a seed such that the paths from $v_1$, $v_2$ merge first (among those from other nodes in the same seed). This gives that $v_2= \n_G(v_1)$ and $v_2 = \n_G(v_1)$.
\begin{figure}
[h]
\begin{tikzpicture}[scale=1]

  \draw[->] (-0.2,0) -- (10.5,0) node[right] {$z$};
  \draw[->] (0,-0.2) -- (0,10.5) node[above] {$V$};

  \foreach \x in {1,2,3,4,5,6,7,8,9,10}
    \draw[shift={(\x,0)}] (0pt,2pt) -- (0pt,-2pt) node[below] {$\x$};

  \foreach \ix/\iy/\ox/\oy/\weight/\where in {1/1/1/2/1/0.8, 1/2/1/3/2/0.8, 1/3/2/4/5/1.3, 1/2/2/4/6/1.7, 2/4/2/5/3/2.2, 2/5/2/6/4/2.2, 2/6/1/7/6/1.7, 1/7/1/3/7/0.8, 1/7/2/8/9/1.3, 2/8/2/9/7/2.2,2/9/2/10/8/2.2}
   \draw (-\ix, \iy) -- (-\ox, \oy) node at (-\where, {(\iy+\oy)/2}) {$\weight$};

  \foreach \x/\y in {1/1,1/2,1/3,2/4,2/5,2/6,1/7,2/8,2/9,2/10}
    \filldraw[white] (-\x,\y) circle (8 pt);
     \foreach \x/\y in {1/1,1/2,1/3,2/4,2/5,2/6,1/7,2/8,2/9,2/10}
    \draw (-\x,\y) circle (8 pt) node {$\y$};

   \draw (-0.2,1) -- (1,1);
   \draw (-0.2,2) -- (1,2);
   \draw (1,1) -- (1,2);
   \draw (1,1.5) -- (2,1.5);
   \draw (-0.2,3) -- (2,3);
   \draw (2, 1.5) -- (2, 3);
   \filldraw (2, {(3+1.5)/2}) circle (2 pt);
   \draw[very thick] (2, {(3+1.5)/2}) -- (5,  {(3+1.5)/2});
   \draw (-0.2,4) -- (3,4);
   \draw (-0.2,5) -- (3,5);
   \draw (3,4) -- (3,5);
   \draw (3,4.5) -- (4,4.5);
   \draw (-0.2,6) -- (4,6);
   \draw (4,4.5) -- (4,6);
   \filldraw (4, {(6+4.5)/2}) circle (2 pt);
   \draw [very thick](4,{(6+4.5)/2}) -- (5,{(6+4.5)/2});
   \draw[very thick] (5, {(6+4.5)/2}) -- (5, {(3+1.5)/2});
   \draw[very thick] (5, {(6+4.5+3+1.5)/4}) -- (6, {(6+4.5+3+1.5)/4});
   \draw[very thick] (6, {(6+4.5+3+1.5)/4}) -- (6,{(6+4.5+3+1.5)/8 + 7/2});
   \draw (6, {(6+4.5+3+1.5)/8 + 7/2}) -- (6, 7);
   \draw [very thick] (6, {(6+4.5+3+1.5)/8 + 7/2}) --(9, {(6+4.5+3+1.5)/8 + 7/2}) ; 
   \draw (-0.2,7) -- (6,7);
   \draw (-0.2,8) -- (7,8);
   \draw (-0.2,9) -- (7,9);
   \draw (7,8) -- (7, 9);
   \draw (7,8.5) -- (8, 8.5);
   \draw (-0.2,10) -- (8,10);
   \draw (8,8.5) -- (8,10);
   \filldraw (8,{(8.5+10)/2})  circle (2 pt);
   \draw[very thick] (8,{(8.5+10)/2}) -- (9,{(8.5+10)/2});
   \draw[very thick] (9,{(8.5+10)/2}) -- (9, {(6+4.5+3+1.5)/8 + 7/2}) ;
	\draw[very thick] (9, { (8.5+10)/4+ (6+4.5+3+1.5)/16 + 7/4}) -- (10.5, { (8.5+10)/4+ (6+4.5+3+1.5)/16 + 7/4});

\end{tikzpicture}
\caption{A weighted graph $G$  and its tree-like representation of the connected components of $G(z)$. Black dots correspond to seeds for the construction of $C_i$'s with $k=3$. Notice that regardless whether the node $7$ is added to $\cur{1,2,3}$ or $\cur{4,5,6}$, the resulting (different) trees have always total weight $5+9=14$.} \label{fig:prim}
\end{figure}
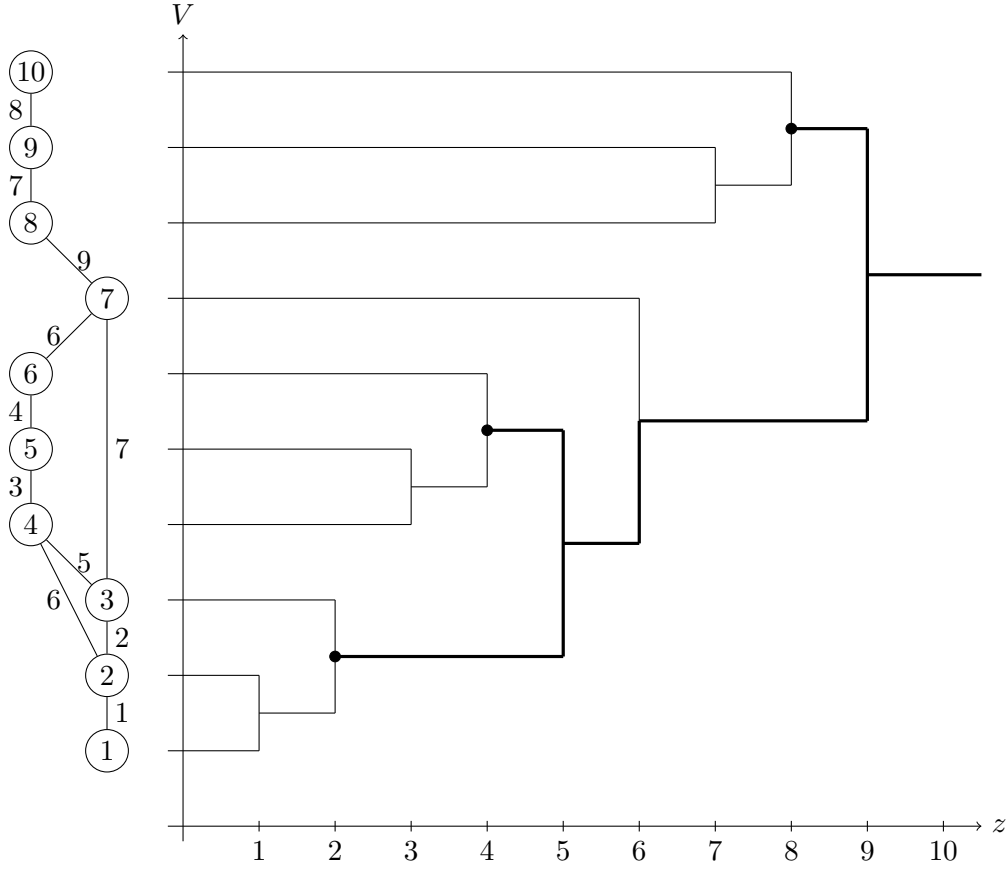

To completely determine every $C_i$, it is simpler to argue graphically on the the tree-like representation (Fig.~\ref{fig:prim}), where we highlight the ``birth'' of $C_i$ at $z_i \in J_{-}$ by thickening the shortest path from the seed towards the root at $z =\infty$. At every $z$ such that two thick paths merge, the corresponding two connected components with at least $k$ elements become one, hence $C_{k,G(z)}$ jumps downwards, i.e., $z \in J_{-}$. For $v \in V$, to determine the set $C_i$ such that $v$ belongs, consider the shortest path towards from the trivial component containing only $v$ at $z=0$ towards the root at $z = \infty$ in the tree-like representation. Let $z_v$ be the smallest value such that such path merges with a thick one (the case that $z_v \in J_+$ and $v$ becomes part of a seed is trivial). We then add to a single $C_i$, among those further from the root on such thick line, all the connected component to which $v$ belongs in $G(z_v^-)$. In fact, a precise choice is not relevant to the thesis, so we simply add it to the $C_i$ with smallest $i$ (other choices may give more desirable properties, but we do not need them for our purposes). 

To prove \eqref{eq:ab-bound-ckg} it is sufficient to realize that the graphical representation of Kruskal's algorithm on the graph $G_k$ gives exactly the thickened tree. \qedhere

\end{proof}
\end{lemma}

\renewcommand{\dist}{\mathsf{d}}
\subsection{Metric MST problem}
If $(X, \dist)$ is a metric space and $V \subseteq X$, then a natural choice for a weight is $w(\cur{x,y}) = \dist(x,y)^p$, where $p>0$ is fixed. If $V$, $R$, $B \subseteq X$, are finite sets and $p>0$, we write
\[ \C^p(V) \quad \text{and respectively} \quad \C^p(R, B),\]
for the the MST cost functional on the complete graph on $V$ (and respectively, on the complete bipartite graph with independent sets $R$, $B$) and edge weights $w(\cur{x,y}) = \dist(x,y)^p$, for $\cur{x,y} \in E$. Notice that, by Remark~\ref{rem:p-dependence-mst}, the MST does not in fact depend on the choice of $p$, and moreover we may let $p\to \infty$ and obtain 
\[ \bra{ \C^p(V)}^{1/p} \to \C^\infty(V), \quad \bra{ \C^p(R, B)}^{1/p} \to   \C^\infty(R,B),\]
where $\C^\infty$ is the minimum bottleneck spanning tree cost defined in \eqref{eq:min-bottleneck} with edge weight given by the distance. 

We denote by
 \[ \dist(V,x) = \inf_{v \in V} \dist(v,x)\]
 and
\[ \dist(R, B) = \sup \cur{ \sup_{r \in R} \inf_{b \in B} \dist(r,b) , \sup_{b \in B} \inf_{r \in R} \dist(r,b) }\]
respectively the distance function from $V$ and the  Hausdorff distance between $R$ and $B$. Clearly, $\C^p(R \cup B) \le \C^p(R, B)$. The following lemma provides a sort of converse inequality.

\begin{lemma}\label{lemma:boundbi}
Let  $p>0$. There exists a constant $C = C(p)>0$ such that, for finite sets $R$, $B \subseteq X$,
\[  \C^p(R, B) \le C \bra{ \C^p(R) +  \sum_{r \in R}  \dist(B,r)^p + \sum_{b \in B} \dist(R,b)^p},\]
and, for some constant $C>0$,
\[ \C^\infty(R,B) \le C\bra{ \C^\infty(R) + \dist(R,B)}.\]
\begin{proof}
For simplicity, we assume that all edge weights are different (otherwise a small perturbation of the weights and a suitable limit gives the thesis).  Let $T_R$ denote the MST for the vertex set $R$ and fix $\bar{r} \in R$. For every $r \in R$, there exists a unique path in $T_R$ with minimal length connecting $r$ to $\bar{r}$. We associate to every $r \in R\setminus\cur{\bar {r}}$ the first edge $e(r)$ of such path (so that $r \in e(r)$). The correspondence $r \mapsto e(r)$ is a bijection.

We use this correspondence to define a connected spanning graph $S$ (not necessarily a tree) on the bipartite graph with independent sets $R$, $B$. For every $r \in R \setminus \cur{\bar{r}}$, if $e(r) = \cur{r, r'}$, we  add the edge $\cur{\n(r), r'}$ to $S$, where $\n(r) = \n_G(r) \in B$ and $G$ is the complete bipartite graph with independent sets $R$, $B$. Moreover, for every $r \in R$, $b \in B$ we also add the edges $\cur{r, \n(r)}$, $\cur{b, \n(b)}$. $S$ is connected because every $b \in B$ is connected to $R$ and the vertex set $R$ is connected: any path on $T_R$ naturally corresponds to a path on $S$ using the pair of edges $\cur{r, \n(r)}$, $\cur{\n(r), r'}$ instead of an edge $e(r) = \cur{r,r'}$. The triangle inequality gives
\[ \dist(\n(r), r')^p \le \bra{ \dist(\n(r),r) + \dist(r, r')}^p \le C\bra{ \dist(\n(r),r)^p + 
\dist(r, r')^p},\]
for some constant $C = C(p)>0$, hence
\[ \sum_{ \cur{r,b} \in S} \dist(r,b)^p \le C \bra{  \sum_{\cur{r,r'} \in T} \dist(r, r')^p + \sum_{r \in R} \dist(r,\n(r))^p + \sum_{b \in B} \dist(b,\n(b))^p}\]
and the first claim follows. Taking the $p$-th root both sides and letting $p \to \infty$ yields the second inequality.
\end{proof}
\end{lemma}


\begin{remark}\label{rem:general-euclidean}
In the Euclidean setting $X=\R^d$, it is known \cite{steele_growth_1988} that, if $p<d$, there exists a constant $C  =C(d,p)>0$ such that
\begin{equation}\label{eq:linfty-bound} \C^p(V) \le C |V|^{1-p/d},\end{equation}
for any $V \subseteq [0,1]^d$. A similar uniform bound cannot be true in the bipartite case, as simple examples show.
\end{remark} 

A second fundamental difference between the usual Euclidean MST and its bipartite variant is that for the latter its maximum vertex degree does not need to be uniformly bounded by a constant $C = C(d)>0$ (again, examples are straightforward). The following result will be crucial to provide an upper bound in the random case. We say that $Q \subseteq \R^d$ is a cube if $Q = \prod_{i=1}^d (x_i, x_i+a)$ with $x = (x_i)_{i=1}^d \in \R^d$ and $a>0$ is its side length. The diameter of $Q$ is then $\sqrt{d}a$ and its volume $|Q| = a^d$. 

\begin{lemma}\label{lem:empty-cone}
Let $R$, $B \subseteq X$, let $T$ be a MST on the bipartite graph with independent sets $R$, $B$ and edge weight $w(x,y) = \dist(x,y)$ and let $\cur{r,b}  \in T$ with $\delta := \dist(r,b) > \dist(R,B)$. Then $S \cap R = \emptyset$, where
\[ S = \cur{ x \in X\, : \, \text{ $\dist(x,r) < \delta -\dist(R,B)$   and  $\dist(x,b)<\delta$}}.\]
In particular, when $X = [0,1]^d$, $S$ contains a cube $Q\subseteq [0,1]^d$ with volume
\[ |Q| =  \min \cur{ \bra{ \frac{ \delta - \dist(R,B)}{2 \sqrt{d}}}^{d}, 1}.\]
\begin{proof} 
Assume by contradiction that there exists $r' \in S \cap R$, and consider the two connected components of the disconnected graph $T \setminus \cur{r,b}$. If $r'$ is in the same component as $r$, then adding $\cur{r', b}$ to $T \setminus \cur{r,b}$ yields a tree (hence, connected) with strictly smaller cost, since $\dist(r',b) < \delta$, hence a contradiction. If $r'$ is in the same component as $b$, then adding $\cur{\n(r), r'}$ to $T\setminus \cur{r,b}$ again yields a tree with strictly smaller cost, since the triangle inequality gives
\[ \dist(\n(r),r') \le \dist(\n(r),r) + \dist(r,r') < \dist(R,B) + \delta - \dist(R,B) < \delta.\]
To prove the last statement, notice first that by convexity of $[0,1]^d$, the point $x$ on the segment connecting $r$ and $b$  at $(\delta - \dist(R,B))/2$ from $r$ belongs to $[0,1]^d$. Moreover, the open ball centred at $x$ with radius $(\delta - \dist(R,B))/2$ is entirely contained in $S$. Finally, intersection of any ball with radius $u\ge 0$ and center  $x \in [0,1]^d$ contains at least a cube of side length $\min\{ u/\sqrt{d}, 1\}$ (the worst case is in general when $x$ is a vertex of $[0,1]^d$).
\end{proof}
\end{lemma}

\subsection{Probabilistic estimates}

In this section we collect some basic probabilistic bounds on distances between  i.i.d.\ uniformly distributed random variables $(X_i)_{i=1}^n$ on a cube $Q \subseteq \R^d$. Some of these facts are well known, especially for $d=1$, since they are related to order statistics, but we provide here short proofs for completeness. The basic observation is that, for every $x\in \R^d$, $0 \le \lambda \le \bra{ |Q|/\omega_d}^{1/d}$, then 
\begin{equation}\label{eq:steele-lower} P(\abs{x- X_i} > \lambda) = \frac{| Q \setminus B(x, \lambda)| }{|Q|} \ge 1- \omega_d \lambda^d/ |Q|,\end{equation}
hence, by independence,
\[P\bra{ \min_{i=1,\ldots, n}\abs{x- X_i} > \lambda} \ge  \bra{1- \omega_d \lambda^d/ |Q|}^n, \quad \text{for every $x \in \R^d$.}
\]
If $x \in Q$, we also have the upper bound (the worst case being $x$ a vertex of $Q$)
\[ P(\abs{x- X_i} > \lambda) \le  1- \omega_d 2^{-d} \lambda^d/ |Q|,\]
hence,
\begin{equation}\label{eq:steele-dist-xi}
   P\bra{ \min_{i=1,\ldots,n} \abs{ x - X_i} > \lambda} \le \bra{1- \omega_d 2^{-d} \lambda ^d/ |Q|}^n, \quad \text{for every $x \in Q$.}
\end{equation}

Assume $Q = [0,1]^d$. The layer-cake formula $\EE\sqa{Z^p} = \int_0^\infty P(Z> \lambda) p \lambda^{p-1} d \lambda$ yields, for every $p>0$,  existence of a constant $C =C(d,p)>0$ such that, for every $n\ge 1$,
\begin{equation}\label{eq:min-moment-p}   C^{-1} n^{-p/d} \le \sup_{x \in [0,1]^d} \EE \sqa{ \min_{i=1,\ldots,n} \abs{ x - X_i}^p}  \le C n^{-p/d}.\end{equation}

For $A \subseteq [0,1]^d$, write $N(A) = \sum_{i=1}^n I_{A}(X_i)$. Then $N(A)$ has binomial law with parameters $(n, |A|)$. In particular, for every $t>0$, 
\[ \EE\sqa{ t^{N(A)}} = (1+ (t-1)|A|)^n.\]
Markov inequality yields that, letting  $F(t) = \bra{t \log t - t + 1}$ then, for every $t>1$,
\begin{equation}\label{eq:concentration-upper}
 P( N(A) > t n|A| ) \le  \exp\bra{ -n |A|F(t)},
\end{equation}
and, for $t<1$,
\begin{equation}\label{eq:concentration-lower} P(N(A) < t n|A|) \le \exp\bra{- n|A|F(t)}.
\end{equation}

We need some uniform bounds on $N(Q)$ for every cube $Q \subseteq [0,1]^d$ with sufficiently large or small volume. We write for brevity, for $v \ge 0$,
\[ N^*(v) := \sup \cur{ N(Q) \, : \, \text{ $Q \subseteq [0,1]^d $ cube with $|Q| \le v$}},\]
and
\[ N_*(v) := \inf \cur{  N(Q) \, : \, \text{ $Q \subseteq [0,1]^d $ cube with $|Q| \ge v$}}.\]

\begin{lemma}\label{lem:cubes-filled}
Let $(X_i)_{i=1}^n$ be i.i.d.\ uniformly distributed on $[0,1]^d$. For every $v\in (0,1)$,  there exist $C = C(v,d) \ge 1$ such that, for every $n \ge C$, if $t>2^{2d}$, then
\begin{equation} \label{eq:upper-uniform-cubes} P\bra{ N^*(v) > t nv } \le\frac{1}{2^d v}  \exp\bra{ - n v 2^{d} F( t 2^{-2d})}, \end{equation}
while, if $t<2^{-2d}$, 
\begin{equation} \label{eq:lower-uniform-cubes} P\bra{ N_*(v) < t nv } \le  \frac{  2^{2d}}{v} \exp\bra{ - n v 2^{-2d} F( t 2^{2d})}.\end{equation}
\begin{proof}
We prove \eqref{eq:upper-uniform-cubes} first. Let $k \in \mathbb{Z}$ such that 
\[ 2^{-k-2} \le v^{1/d} < 2^{-k-1},\]
 so that every cube $Q$ with $|Q| =v$ is contained in a dyadic cube $\prod_{\ell=1}^d ( n_\ell 2^{-k}, (n_\ell+1)2^{-k})$, with $(n_\ell)_{i=1}^d \in \cur{0,\ldots, 2^{k}-1}^d$. It is then sufficient to consider the event $N(Q) > t nv$ for at least one such dyadic cube, i.e., using the union bound \eqref{eq:concentration-upper}, we bound from above 
\[ P\bra{ N^*(v ) > tnv }\le 2^{kd} P\bra{ N(Q) > t nv ) }.\]
Using that, for a dyadic cube with $|Q| = 2^{-dk}$,
\[ 2^{kd} \le 1/( 2^{d} v ) \quad \text{and} \quad  |Q|2^{-2d} \le v \le 2^{-d} |Q|, \]
it follows that \eqref{eq:concentration-upper} applies with $2^{-2d}t$ instead of $t$, yielding
\[ \begin{split} 2^{kd} P\bra{ N(Q) > t nv } & \le 2^{kd} P\bra{ N(Q) > 2^{-2d} t n|	Q| } \\
& \le \frac{1}{2^d v}  \exp\bra{ - n|Q|  F( t 2^{-2d})}\\
& \le \frac{1}{2^d v}  \exp\bra{ - n v 2^{d} F( t 2^{-2d})}.
\end{split}\]

The argument for \eqref{eq:lower-uniform-cubes} is analogous. Let $k \in \mathbb{Z}$ be such that
\[ 2^{-k+1} \le v^{1/d} < 2^{-k+2},\]
hence every cube $Q$ with $|Q| \ge v$ contains at least one dyadic cube with volume $2^{-kd}$, and we are reduced to consider the event that for such a cube $N(Q) < t nv $, i.e., 
 \[ P\bra{ N_*(v ) < tnv }\le 2^{kd} P\bra{ N(Q) < tnv } .\]
Using that 
\[ 2^{kd} \le 2^{2d}/v \quad \text{and} \quad  2^d |Q|\le v \le 2^{2d}|Q|,  \]
it follows that \eqref{eq:concentration-lower} applies with $2^{2d}t$ instead of $t$, yielding
\[ \begin{split} 2^{kd} P\bra{ N(Q) < t nv } & \le 2^{kd} P\bra{ N(Q) < 2^{2d} t n |Q|} \\
& \le \frac{  2^{2d}}{v} \exp\bra{ - n |Q| F( t 2^{2d})}\\
& \le \frac{  2^{2d}}{v} \exp\bra{ - n v 2^{-2d} F( t 2^{2d})}.\qedhere
\end{split}\] 
\end{proof}
\end{lemma}

In the following result we investigate the random variable (for $n \ge 2$)
\[ M^n = \max_{i=1, \ldots, n} |X_i - \n(X_i)|,\]
where $\n(X_i)$ denotes the closest point to $X_i$ among $\cur{X_j}_{j \neq i}$, so that
\[ |X_i - \n(X_i)| = \min_{\substack{ j \le n \\ j \neq i}} |X_i - X_j|.\]
Heuristically, since $|X_i-\n(X_i)| \sim n^{-1/d}$, and the random variables are almost independent, we still expect that $M^n \sim n^{-1/d}$, up to logarithmic factors. This is indeed the case.


\begin{proposition}\label{prop:max-min}
Let $n \ge 2$, $(X_i)_{i=1}^n$ be i.i.d.~ uniformly distributed on $[0,1]^d$. Then, for every $a>0$, there exists a constant $C = C(a,d) \ge 1$ such that
\[  P \bra{ C^{-1} \le  \frac{M^n}{(\log( n) / n)^{1/d}} \le C } \ge 1 - \frac{C}{n^a}.\]
In particular, for every $q>0$, there exists $C =C(d,q)>0$ such that, for every $n \ge 2$,
\begin{equation} \label{eq:moment-M-n}\EE\sqa{ \bra{ M^n }^q } \le C \bra{ \frac{ \log n}{n}}^{q/d}.\end{equation}
\begin{proof}



%
%
%
%
For every $n$ sufficiently large, we choose $\eta = \eta(d,n) \in ((a+1)2^{2d+1}, (a+1)2^{2d+2})$ such that, defining $\delta =  \bra{ \eta \log( n) / {n}}^{1/d}$, we have that $1/(3\delta)^d$ is integer.
 
 Consider a partition of $[0,1]^d$ into cubes $\cur{Q_j}_{j\in J}$ of volume $|Q_j| = v :=\delta^d$, with $J = \cur{1, \ldots, \delta^{-d}}$. Fix $\bar{t}>2^{2d}$ such that $2^d F(\bar{t}2^{-2d}) > a+1$.  Lemma~\ref{lem:cubes-filled} entails that the event
\[ E := \bigcap_{j \in J} \cur{2 \le N(Q_j)  \le \bar{t} \log n } \supseteq \cur{N^*(v) \le \bar{t} \log n } \cap \cur{N_*(v) \ge 2}.\]
has probability larger than $1- C/n^a$ for some constant $C = C(d,a)$. Indeed, \eqref{eq:upper-uniform-cubes} gives
\[\begin{split}  P( N^*(v) \le \bar{t} \log n ) & = P\bra{N^*(v) \le  \bar{t}/\eta nv} \\
& \le 2^{-d} \frac{n}{\eta \log n} \exp\bra{ - \log(n) \eta  2^d F(\bar{t}2^{2d}/\eta ) } \\
&^\le \frac{1}{2^{d+1} \eta n^a},\end{split}\]
having also used that $\eta \ge 1$. Similarly, by \eqref{eq:lower-uniform-cubes}, with $t = 2/(\log(n) \eta)$,
\[ \begin{split} P( N_*(v) > 2 )  & = P\bra{N_*(v) > \bra{ 2/(\log(n) \eta) } nv} 
\\
& \le 2^{2d} \frac{n}{\eta \log n} \exp\bra{ - \log(n) \eta 2^{-2d} F(2/(\log(n) \eta))}.\end{split} \]
For $n$ sufficiently large, $ F\bra{2/(\log(n) \eta) }o> 1/2$, hence
\[\begin{split}  2^{2d} \frac{n}{\eta \log n} \exp\bra{ - \log(n) \eta 2^{-2d} F(2/(\log(n) \eta))} & \le 2^{2d} \frac{n}{\eta \log n } \exp\bra{ - \log(n) \eta 2^{-2d-1} }\\
& \le \frac{2^{2d}}{\eta n^a}. 
\end{split} \]


 Hence,  to prove the thesis, we can assume that $E$ holds.  In such a case, it follows  at once that $M^n \le \sqrt{d} v^{1/d}$, hence
\[ P\bra{ \frac{M^n}{(\log (n) /n)^{1/d}} > \sqrt{d} \eta^{1/d} }  \le  \frac{C}{n^a}.\]

Next, we introduce the random variables
\[ M^n_j  := \sup_{\substack{ i \le n \\ X_i \in Q_j} }\,  \inf_{ \substack{ \ell \le n \\ 0< |X_i- X_\ell| < \delta } } | X_i - X_\ell|,\]
i.e., we maximize the minimum distances between points in $Q_j$ and those that are at distance at most $\delta$. These are not necessarily in $Q_j$ but must belong to the union of all cubes $Q_\ell$ covering $\cur{x |  \dist(x, Q_j) \le \delta}$, that we denote by $Q_{j, \delta}$. Since  every cube $\cur{Q_j}_{j\in J}$ has side length $\delta$, then $Q_{j,\delta}$  is a cube of side length $3 \delta$, hence  $|Q_{j,\delta}| /|Q| = 3^d$. Notice also that, since each cube $Q_j$ contains at least two elements (for we assume that $E$ holds) then
\[ M^n = \max_{j \in J} M^n_j.\]

We now use the following fact: for every $f: \cur{1, \ldots, n} \to J$, conditioning upon the event 
\[  A_f = \cur{ X_i \in Q_{f(i)}\quad  \text{for every $i=1, \ldots, n$,}}, \]
the random variables $(X_i)_{i=1}^n$ are independent, each $X_{i}$ uniform on $Q_{f(i)}$. Thus, we further disintegrate upon the events $A_f$, and since $E$ holds we consider only $f$'s such that, for every $j\in J$, 
\[ 2 \le \abs{\cur{f = j} } 	\le \bar{t} \log n.\]
%
We let $K \subseteq J$ denote a subfamily consisting of $(\delta 3)^{-d}$ cubes such that, for $j,k \in K$, with $j \neq k$, $Q_{j,\delta} \cap Q_{k,\delta} = \emptyset$ so that the random variables $(M^n_k)_{k \in K}$ are independent (after conditioning upon $A_f$). %
Using 
\[ M^n = \max_{j\in J} M^n_j \ge \max_{k \in K}  M^n_k,\]
and independence we have, for every $\lambda>0$,
\[ 
P(  M^n   \le   \lambda | A_f) \le P( \max_{k \in K}  M^n_k \le \lambda | A_f ) = \prod_{k \in K}  P(   M^n_k \le \lambda  | A_f ) \]
The probability $P( M^n_k \le \lambda | A_f )$ clearly depends only on the number of elements in each $\cur{f = j}$, i.e., on the number of points in each $Q_j$, for $j \in J$, not their labels. We may therefore assume that $f(1) = k$, i.e., $X_1 \in Q_k$ and that
\[ \cur{X_i}_{i=1}^n \cap Q_{k,\delta} = \cur{X_1, \ldots, X_\ell},\]
with
\begin{equation}\label{eq:n-qk-delta-upper} \ell = N(Q_{k, \delta})\le 3^d \bar{t} \log n.\end{equation}
 Then, 
\[ M^n_k \ge \min_{ j=2,\ldots, \ell} |X_j - X_1|,\]
hence, further conditioning upon $X_1$ and using independence,
\[ \begin{split} P(M^n_k > \lambda  | A_f ) & \ge \frac{1}{|Q_k|} \int_{Q_k} P\bra{ \min_{  j=2,\ldots, \ell} |X_j - x| >\lambda  | A_f, X_1=x} d x \\ 
& \ge \frac{1}{|Q_k|} \int_{Q_k}  \prod_{j=2}^\ell P( |X_j - x| >\lambda | A_f, X_1=x) d x\\
& \ge \frac{1}{|Q_k|} \int_{Q_k}  \prod_{j=2}^\ell \bra{ 1- \omega_d \lambda^d/ |Q_{f(j)}|}_+ d x \quad \text{by \eqref{eq:steele-lower},}\\
& \ge \bra{ 1- \omega_d (\lambda/\delta)^{d} }^{N(Q_{k,\delta})-1}_+ =  e^{-u ( N(Q_{k,\delta})-1)},
\end{split} \]
where in the last equality we choose $\lambda = \delta ((1-e^{-u})/ \omega_d)^{1/d}$ with $u =1/(2\cdot 3^d \bar{t})$.  Indeed, this choice ensures that, by \eqref{eq:n-qk-delta-upper} we bound from above, for $n$ sufficiently large, 
\[ P\bra{ M^n \le  \lambda| A_f }  \le \bra{ 1-e^{-u ( 3^d \bar{t} \log(n)-1)}}^{(\delta 3)^{-d}} \le \bra{ 1-\frac{e^{u}}{n^{1/2}}}^{n/ (3^{d}\log(n))} \le \frac{C}{n^a}.
\]
Finally, \eqref{eq:moment-M-n} follows since $M^n \le \sqrt{d}$, hence, choosing $a = q/d$, we bound from above
\[ \EE\sqa{ \bra{M^n}^q } \le d^{q/2}  P( M^n > C (\log(n)/n)^{1/d} ) + C^q \bra{ \frac{ \log n}{n}}^{q/d}\le C'\bra{ \frac{ \log n}{n}}^{q/d} . \qedhere\]
\end{proof}
\end{proposition}

%
%

A minor variation of the proof of the previous proposition yields the following bipartite analogue, where we replace $M^n$ with the (random) Hausdorff distance between $R$ and $B$.

\begin{proposition}\label{prop:max-min-bipartite}
For $n \ge 1$, let $R^n = \cur{X_i}_{i=1}^{n_R}$, $B^n= \cur{Y_i}_{i=1}^{n_B}$ be (jointly) i.i.d.~uniformly distributed on $[0,1]^d$ with $n_R+n_B = n$ and
\[ \lim_{n \to \infty} \frac{n_R}{n} = \alpha_R \in (0,1), \quad \lim_{n \to \infty} \frac{n_B}{n} = \alpha_B = 1 -\alpha_R.\]
Then, for every $a>0$, there exists a constant $C = C(d,\alpha_R, \alpha_B, a)\ge 1$ such that, for every $n$ sufficiently large,
\[  P \bra{ C^{-1} \le  \frac{\dist(R^n, B^n)}{(\log( n) / n)^{1/d}} \le C } \ge 1 - \frac{C}{n^a}.\]
In particular, for every $q>0$, there exists $C =C(d,q,\alpha_R)>0$ such that, for every $n$ sufficiently large,
\[ \EE\sqa{  \dist(R^n, B^n ) ^q } \le C \bra{ \frac{ \log n}{n}}^{q/d}.\]
\end{proposition}

\section{Proof of Theorem~\ref{thm:main-degree}}\label{sec:main-degree}

Throughout this section, for $n \ge 1$, let $R^n = \cur{X_i}_{i=1}^{n_R}$, $B^n= \cur{Y_i}_{i=1}^{n_B}$ be (jointly) i.i.d.~uniformly distributed on $[0,1]^d$ with $n_R+n_B = n$ and
\[ \lim_{n \to \infty} \frac{n_R}{n} = \alpha_R \in (0,1), \quad \lim_{n \to \infty} \frac{n_B}{n} = \alpha_B = 1 -\alpha_R.\]
Let $T^n$ denote the Euclidean bipartite MST on $R^n$, $B^n$, 
and write $\Delta(T^n)$ be the maximum vertex degree of $T^n$.

We split the proof into two separate results.

\begin{lemma}
There exists $\epsilon>0$ such that, as $n \to \infty$, 
\[ P\bra{ \Delta(T^n) < \epsilon \log n} \to 0.\]
\begin{proof} For $n$ sufficiently large, we have $n_R \ge\alpha_R n /2$, $n_B \ge \alpha_B n /2$. Fix any $a>0$ and let $C_1 = C(d,a)\ge 1$ be as in Proposition~\ref{prop:max-min} applied to the variables $(X_i)_{i=1}^{n_R}$, so that the event $E^n$, such that there exists $X_i \in R^n$ with
\[ \min_{j \neq i} |X_i - X_j| \ge C^{-1}_1 \bra{ \log(n_R)/n_R}^{1/d} \ge C^{-1} \bra{ \log(n)/n }^{1/d},\]
(the second inequality being true if $n$ is sufficiently large) has probability $P(E^n)\ge 1-C/n^a$ as $n \to \infty$, for a suitable constant $C = C(d,\alpha_R, a)>0$.
In particular, every point in the cube $Q$ centred at $X_i$ with side length
\[ \delta = \bra{ \log(n)/n }^{1/d} / (4 \sqrt{d} C)\]
is strictly closer to $X_i$ than any other point in $R^n$. Notice that $E^n$ and such cube $Q$ depend on the random variables $R^n$ only. By Lemma~\ref{lem:voronoi}, if $E^n$ holds, every $Y_j \in Q$ will be adjacent in $T^n$ to $X_i$, hence, choosing 
\[ \epsilon = 1/\bra{ 2(4\sqrt{d} C)^d}, \quad \text{ so that } \quad (4 \sqrt{d} C)^d \epsilon=1/2,\]
 and writing $N_{B^n}(Q)$ for the number of elements in $B^n \cap  Q$, we have
\[ \begin{split} P( \Delta(T_n) <  \epsilon \log(n) ) & \le P((E^n)^c) + P( N_{B^n}(Q) < n|Q|/2, E^n) \\
& \le \frac{C}{n^a} + \EE\sqa{ I_{E^n} P \bra{N_{B^n}(Q) < n \delta^d/2 | \cur{X_i}_{i=1}^{n_B} }}.
\end{split}\]
By independence, the conditional law of $N_{B^n}(Q)$ is Binomial with parameters $\bra{ n_B, \delta^d}$, hence we may use \eqref{eq:concentration-lower} to obtain
\[ P \bra{N_{B^n}(Q) < \epsilon \log(n) | \cur{X_i}_{i=1}^{n_B} } \le \exp\bra{ -n_B \delta^d F(1/2)}.\]
Since $n_B \ge \alpha_B n /2$, $\delta^d = \log(n)/ (n (4 \sqrt{d} C))$ and $F(1/2)>0$, it follows that, as $n \to \infty$, 
\[ \exp\bra{ -n_B \delta^d F(1/2)} \to 0.\qedhere\]

\end{proof}
\end{lemma}

\begin{lemma}\label{lem:lower-bound-max-deg}
For every $a>0$, exists $C = C(d,a, \alpha_R) >0$ such that, for $n$ sufficiently large, 
\[ P\bra{ \Delta(T^n) > C \log(n) } \le \frac{C}{n^a}.\]
In particular, for every $q>0$, there exists $C =C(d, q, \alpha_R)>0$ such that, for every $n$ large enough,
\begin{equation} \label{eq:moment-delta-tn}\EE\sqa{ \Delta(T^n) ^q } \le C \bra{ \log n}^q.\end{equation}
\begin{proof}
For $n$ sufficiently large, we have $n_R \ge\alpha_R n /2$, $n_B \ge \alpha_B n /2$. Fix $a>0$ and let $C_1 = C_1(d,a, \alpha_R)\ge 1$ be as in Proposition~\ref{prop:max-min-bipartite} applied to the variables $(X_i)_{i=1}^{n_R}$, $(Y_j)_{j=1}^{n_B}$ so that, if $n$ is sufficiently large, the event 
\[  \cur{ \dist(R^n,B^n) > C_1 (\log(n) /n)^{1/d} }\]
has probability smaller than $C_1/n^a$. We claim that there exist $\lambda=\lambda(d,a, \alpha_R) > 8 C_1$ and $C = C(d,a,\alpha_R)$ such that the following conditions hold, for $n$ sufficiently large:
\begin{enumerate}
\item letting $v_1 = \bra{\lambda/8}^d \log(n)/n$, then
\[ P\bra{ N^*_{R^n}(v_1) >  \lambda^d \log n } \le \frac{C}{n^a} \quad \text{and} \quad P\bra{ N^*_{B^n}(v_1) >  \lambda^d \log n } \le \frac{C}{n^a};\]
\item letting $v_2 = \sqa{ \bra{ \lambda/8 - C_1)}/(\sqrt{d})}^{d}  \log(n)/n$, then
\[ P\bra{ \bra{N_{R^n}}_*(v_2) =0 } \le  \frac{C}{n^a}, \quad P\bra{ \bra{N_{B^n}}_*(v_2) =0 } \le  \frac{C}{n^a}.
\]
\end{enumerate}
Once the claim is proved, it is immediate to show that on the event
\[\begin{split} 
E = \Big\{ \dist(R,B) > C_1 (\log(n) /n)^{1/d}, & \max\cur{  N^*_{R^n}(v_1),  N^*_{B^n}(v_1)} \le \lambda^d \log n \\
& \min\cur{  \bra{N_{R^n}}_*(v_2), \bra{N_{B^n}}_*(v_2)} \ge 1 \Big\}
\end{split}\]
that satisfies $P(E^c) < C/n^a$ (with a different constant $C$), it must be $\Delta(T^n) \le \lambda^d \log(n)$. Indeed, if $\Delta(T^n) > \lambda^d \log n$ assuming without loss of generality that $X_i \in R^n$ has degree larger than $\lambda^d \log n$, it follows from $N_{B^n}^*(v_1) \le \lambda^d \log n$ that there must be a node in $Y_j \in B^n$ adjacent to $X_i$ that does not belong to the cube with volume $v_1$ centered at $X_i$. In particular,
\[ |Y_j-X_i| \ge \frac{\lambda}{8} \bra{ \log(n)/n}^{1/d} > C_1 \bra {\log(n)/n}^{1/d} \ge \dist(R^n,B^n).\]
By Lemma~\ref{lem:empty-cone}, there exists a cube $Q$ with $Q \cap R^n = \emptyset$ and
\[ |Q| \ge \sqa{ \bra{ \lambda/8 - C_1)}/ (2 \sqrt{d})}^{d}  \frac{\log(n)}{n} = v_2,\]
contradicting $(N_{R^n})_*(v_2) >0$.

To prove the claim, we apply Lemma~\ref{lem:cubes-filled}. Indeed, \eqref{eq:upper-uniform-cubes} with $t= 8^d n/n_R$ gives
\[\begin{split} P\bra{ N^*_{R^n}(v_1) >  \lambda^d \log n } & = P\bra{ N^*_{R^n}(v_1) >  8^d (n/n_R)  \cdot v_1 n_R  } \\
& \le \frac{1}{2^d v_1} \exp\bra{-n_R v_1 2^dF(2^dn/n_R)} \\
& \le \frac{2^{2d} n }{\lambda^d \log(n)} \exp\bra{ - \log(n) (n_R/n) \cdot \lambda^d2^{-2d}F(2^d)} \le  \frac{C}{n^a},
\end{split}\]
provided that $n$ is sufficiently large and $\lambda^d \alpha_R 2^{-2d-1}F(2^d) \ge a+1$. Conversely, we use \eqref{eq:lower-uniform-cubes} with $t = 1/n_R v_2 = \sqa{ \bra{ \lambda/8 - C_1)}/( \sqrt{d})}^{-d} n/(n_R\log(n)) < 2^{-2d}$ (if $n$ is sufficiently large) so that
\[  P\bra{ N^*_{R^n}(v_1) < 1 } \le  2^{2d} \sqa{ \bra{ \lambda/8 - C_1)}/(2 \sqrt{d})}^{-d} \frac{n}{\log n} \exp\bra{ -n_R v_2 2^{-2d} F(t2^{2d})} \le \frac{C}{n^a},\]
provided that we choose $\lambda$ large enough such that 
\[ \sqa{ \bra{ \lambda/8 - C_1)}/(2 \sqrt{d})}^{d}2^{-2d-1} \ge a+1,\]
since for $n$ sufficiently large, we have $F(t2^{2d}) > 1/2$. 

Finally, \eqref{eq:moment-delta-tn} follows since trivially $\Delta(T^n) \le n$, hence, choosing $a = q$, we bound from above
\[ \EE\sqa{  \Delta(T^n)^q } \le n^q  P( \Delta(T^n) > C \log n) + \bra{ C \log n}^q \le C+ \bra{ C \log n}^q. \qedhere\]
\end{proof}
\end{lemma}

\section{Proof of Theorem~\ref{thm:main-bipartite-euclidean}}\label{sec:avram-berstimas}

We first extend \cite[Theorem 1]{avram_minimum_1992} to the bipartite case. Let $G =(V, E, w)$ be a random weighted graph, i.e., $(w_e)_{e \in E}$ are random variables. To simplify, we assume throughout this section that $G$ is the complete graph over $V =\cur{1,\ldots, m}$ for some $m \ge 1$, but allow weights $w(e) \in [0,\infty]$. Recall that connection between nodes is meant only along paths consisting of edges with finite weight: we assume in particular that $G$ is a.s.\ connected. 
The number of connected components of $G(z)$ can be written as
\begin{equation}\label{eq:CGz}
 C_{G(z)}= \sum_{k=1}^m \sum_{i=1}^{m}\frac{X_{i,k,G(z)}}{k},
 \end{equation}
where the random variable $X_{i,k,G(z)} \in \cur{0,1}$ indicates whether $i \in V$ belongs to a component of $G(z)$ having exactly $k$ elements. Similarly, for the number of connected components having at least $k$ nodes,
\[  C_{k,G(z)}= \sum_{\ell=k}^m \sum_{i=1}^{m}\frac{X_{i,\ell,G(z)}}{\ell},\]

To estimate the expectation of $X_{i,k,G(z)}$, in \cite{avram_minimum_1992} it is assumed that $\bra{w_{i,j}}_{i,j \in V}$ are exchangeable random variables. To extend the validity of their results to the bipartite case we relax this condition by requiring that the joint law of the weights is invariant with respect to the symmetries of an underlying graph model (such as a complete bipartite graph). Let us give the following general definition. 

\begin{definition}
On a random weighted graph $G = (V, E, w)$, nodes $i, j \in V$ are said equivalent in law if there exists a bijection $\pi: V \to V$ such that $\pi(i) = j$ and $\bra{w_{k,\ell}}_{k,\ell\in V}$ have the same joint law as $\bra{w_{\pi(k),\pi(\ell)}}_{k,\ell\in V}$.
\end{definition}

Clearly, this defines an equivalence relation, which is relevant for our purposes since, if $i, j \in V$ are equivalent in law, then for every $k$, $z \ge 0$ the random variables $X_{i,k,G(z)}$, $X_{j,k,G(z)}$ have the same law. Therefore, when computing the expectation $\EE\sqa{C_{G(z)}}$ using \eqref{eq:CGz}, we are reduced to a summation upon $k$ and the equivalence classes. If the weights are exchangeable, then there is only one equivalence class, but this is also the case  a random Euclidean bipartite graph with $V = R\cup B$ and $|R| = |B|$. 
To deal with bipartite graphs with $|R| \neq |B|$ we consider the case of two (non empty) equivalence classes $R$ and $B$. We introduce the functions  
\[  P_{k, R, G(z)}= \EE\sqa{ X_{i,k,G(z)}} \quad \text{for $i \in R$, and }\quad   P_{k, B,G(z)}= \EE\sqa{ X_{j,k,G(z)}} \quad \text{for $j \in B$,}\]
i.e., the probability that a given node in $R$ (respectively in $B$) belongs to a connected component of $G(z)$ with exactly $k$ elements.
Taking the expectation in \eqref{eq:ab-general} and \eqref{eq:CGz}, we deduce that
\begin{equation}\label{eq:ee-ab-general} \EE\sqa{ \C(G) } = \int_0^\infty \bra{ \sum_{k=1}^{|V|} \frac{1}{k} \bra{ |R|  P_{k, R, G(z)} + |B|  P_{k,B, G(z)}} -1  }d z .\end{equation}




Consider now a sequence of random graphs $\bra{G^n}_{n=1}^\infty = \bra{(V^n, E^n, w^n)}_{n=1}^\infty$, each with two equivalence classes $V^n = R^n \cup B^n$, and write, for brevity,
\[ C^n(z)= C_{G^n(z)}, \quad C^n_{k}(z) = C_{k,G^n(z)},\]
\[ P_{k,R}^n(z) = P_{k,  R, G^n(z)} \quad \text{and} \quad P_{k,B}^n(z) = P_{k, B, G^n(z)}.\]
 For a  (pseudo-dimension) parameter  $d >1$ we introduce the following assumptions:
\begin{enumerate}[a)]
\item \label{ab-hp-1}For any $k\ge 1$, $y>0$,
\[ \lim_{n \to \infty } P_{k,R}^n\bra{ (y/n)^{1/d} } = f_{k,R}(y), \quad \text{and} \quad \lim_{n \to \infty } P_{k,B}^n\bra{ (y/n)^{1/d} } = f_{k,B}(y),\]
where convergence is pointwise and dominated in the following sense: there exists a function
$\ell_{k}$ such that, for every $y>0$,
\[ \sup_{n} P_{k,R}^n\bra{ (y/n)^{1/d} }+ P_{k,B}^n\bra{ (y/n)^{1/d} } \le \ell_k(y)\]
and 
\[ \int_{0}^{\infty} \ell_{k}(y)y^{\frac{1}{d}-1}dy < \infty.\]
%
%

\item \label{ab-hp-2} It holds 
\[
 \lim_{k \to \infty} \limsup_{n\to \infty
 } \abs{ \frac{1}{n^{1-\frac{1}{d}}} \int_{0}^{\infty} \bra{ \EE\sqa{C_{k}^n(z)}-1} dz} = 0.\]
\end{enumerate}


The following result extends \cite[Theorem 1]{avram_minimum_1992} to the bipartite case.
\begin{theorem}\label{theoBert}
Let $d >1$ and $\bra{G^n}_{n=1}^{\infty}$ be a sequence of random graphs, each with two equivalence classes $V^n = R^n \cup B^n$, satisfying assumptions \ref{ab-hp-1}), \ref{ab-hp-2}), and
\[ \lim_{n \to \infty} \frac{|R^n|}{n} = \alpha_R\in (0,1) \quad\text{and} \quad \lim_{n \to \infty} \frac{|B^n|}{n} = \alpha_B  = 1-\alpha_R.\]
Then, 
\[
\lim_{n \to \infty} \frac{\EE\sqa{\C(G^n)}}{n^{1-1/d}} = \frac{1}{d}\sum_{k=1}^{\infty} \frac{1}{k}\int_{0}^{\infty}\bra{ \alpha_Rf_{k,R}(y)+ \alpha_B f_{k,B}(y)}y^{\frac{1}{d}-1}dy.
\]
\end{theorem}

\begin{proof} 
By \eqref{eq:ab-general},  for any $n \ge 1$, $k \ge 1$, we decompose
\[\begin{split}
 \frac{\EE\sqa{\C(G^n)}}{n^{1-1/d}} & = \frac{1}{n^{1-1/d}} \int_{0}^{\infty}\bra{\sum_{k=1}^{|V^n|}\frac{1}{k} \bra{ |R^n| P_{k, R}^n(z) +  |B^n| P_{k, B}^n(z)}-1}dz\\
& =   n^{1/d}  \int_{0}^{\infty} \bra{\sum_{\ell=1}^{k-1}\frac{1}{\ell} \bra{ \frac{|R^n|}{n} P_{\ell, R}^n(z) +  \frac{|B^n|}{n} P_{\ell, B}^n(z)}-\frac{1}{n}}dz  \\
& \quad + \frac{1}{n^{1-1/d}} \int_{0}^{\infty} \bra{ \EE\sqa{C_{k}^n(z)}-1} dz.
\end{split}\]
Assumption \ref{ab-hp-2}) gives that in the limit $n \to \infty$, $k \to \infty$ the second term gives no contribution. Hence, it is sufficient to let $n \to \infty$ and then $k\to \infty$ in first term. Actually, since $n^{1-1/d} \to \infty$, we only need to prove that
\[\begin{split} & \lim_{k \to \infty} \lim_{n\to \infty} n^{1/d}  \int_{0}^{\infty} \sum_{\ell=1}^{k-1}\frac{1}{\ell} \bra{ \frac{|R^n|}{n} P_{\ell, R}^n(z) +  \frac{|B^n|}{n} P_{\ell, B}^n(z)}dz \\
& \quad = \frac{1}{d}\sum_{k=1}^{\infty} \frac{1}{k}\int_{0}^{\infty}\bra{ \alpha_Rf_{k,R}(y)+ \alpha_B f_{k,B}(y)}y^{1/d-1}dy,\end{split}\]
which follows by dominated convergence, because of assumption \ref{ab-hp-2}), after the change of variables $z = (y/n)^{1/d}$. 
\end{proof}
\renewcommand{\T}{\mathbb{T}}

We now apply the above theorem to the bipartite MST problem on the $d$-dimensional flat torus $\T^d = \R^d/\mathbb{Z}^d$, endowed with the flat distance 
\[ \dist_{\T^d}(x,y) := \inf_{z \in\mathbb{Z}^d} |x-y+z|.\] 

\begin{theorem}\label{formula-torus}
Let $d \ge 1$, $n\ge 1$, $R^n = \cur{X_i}_{i=1}^{n_R}$, $B^n = \cur{Y_i}_{i=1}^{n_B}$ be (jointly) i.i.d.~uniformly distributed on $\T^d$ with $n_R+n_B = n$ and
\[ \lim_{n \to \infty} \frac { n_R }{n} = \alpha_R, \quad \lim_{n \to \infty} \frac { n_B}{n} = \alpha_B = 1-\alpha_R.\]
Then, for every $p \in (0,d)$,
\[ \lim_{n \to \infty} \frac{ \EE\sqa{ \C^p(R^n,B^n)}}{ n^{1-p/d} }= \beta_{bMST}(d,p),\]
with $\beta_{bMST}(d,p)$ as in Theorem~\ref{thm:main-bipartite-euclidean}.


\begin{proof}
We apply Theorem~\ref{theoBert} to the random bipartite graph $G^n$ over $V^n=R^n\cup B^n$, and $w^n(X_i, Y_j) = \dist_{\T^d}(X_i, Y_j)^p$ -- that can be naturally identified with a graph over $\cur{1,\ldots, n}$ (recall that we allow for infinite weights). We show separately that assumptions \ref{ab-hp-1}) and \ref{ab-hp-2}) hold with $d/p>1$ instead of $d$.


We introduce some notation: for $k_R$, $k_B \in \mathbb{N}$, $z\ge 0$, let 
\[ \Theta_{\T^d}(k_R, k_B, z) \subseteq (\T^d)^{k_R} \times (\T^d)^{k_B}\]
 denote the set of (ordered) points $( \bra{r_i}_{i=1}^{k_R} , \bra{b_j}_{j=1}^{k_B})$ such that, in the associated bipartite graph with weights  $( \dist_{\T^d}(r_i,b_j)^p)_{i,j}$, the subgraph with all edges having weight less than $z$ is connected (or equivalently, there exists a bipartite spanning tree having all edges with weight less than $z$). For a set $A \subseteq \T^d$, $z\ge 0$, write 
 \[ D_{\T^d}(A,z) = \cur{x \in \T^d\, : \, \dist_{\T^d}(A,x)^p  \le z}.\]




Recall that by definition $P_{k,n,R}(z)$ is the probability that a fixed vertex in $R$, say $X_1$, belongs to a component of the subgraph $G^n(z)$ having exactly $k$ nodes. We disintegrate upon the nodes in $R$ and in $B$ belonging to such component. Clearly, only their numbers are relevant, not the precise labels (except for $X_1$ that is fixed). Therefore, we compute the probability $I_R(k_R, k_B, z)$ that $X_1$ belongs to a component with $k_R$ nodes $\cur{X_i}_{i=1}^{k_R} \subseteq R$ and $k_B$ nodes $\cur{Y_j}_{j=1}^{k_B} \subseteq B$, with $k_R+k_B = k$, which is precisely described as follows:
\begin{enumerate}[i)]
\item  the set $( \cur{X_i}_{i=1}^{k_R}, \cur{Y_j}_{j=1}^{k_B})$  belongs to $\Theta_{\T^d}(k_R, k_B, z)$,
\item \label{e-1}$\dist_{\T^d}( \cur{Y_j}_{j=1}^{k_B}, X_i ) >z$, i.e. $X_i \notin D_{\T^d} ( \cur{Y_j}_{j=1}^{n_B}, z)$, for every $i > k_R$,
\item \label{e-2} $\dist_{\T^d} (\cur{X_i}_{i=1}^{k_R},  Y_j) >z$, i.e. $Y_j \notin D_{\T^d} \bra{\cur{X_i}_{i=1}^{k_R}, z}$, for every $j > k_B$.
\end{enumerate}
Conditioning upon $\cur{X_i}_{i=1}^{k_R}$, $\cur{Y_j}_{j=1}^{k_B}$ and using independence for the events \ref{e-1}) and \ref{e-2}), we obtain the following expression for the probability:
\begin{equation}\begin{split} \label{eq:I-torus} & I_R(k_R, k_B, z) =\\
&  \int_{\Theta_{\T^d}(k_R, k_B, z)}  \bra{ 1 - |D_{\T^d}(\cur{b_j}_{j=1}^{k_B}, z)|}^{n_R - k_R}\bra{ 1 - |D_{\T^d} (\cur{r_i}_{i=1}^{k_R}, z)|}^{n_B - k_B} d r d b, \end{split} \end{equation}
where $d r d b$ stands for integration performed with respect to the variables $\cur{r_i}_{i=1}^{k_R}$ and $\cur{b_j}_{j=1}^{k_B}$. Summing upon all the ${n_R \choose k_R-1} {n_B \choose k_B}$ different choices of labellings (recall that $X_1$ is kept fixed) and upon $k_R \ge 1$, $k_B$ with $k_R+k_B = k$ gives
\begin{equation}\label{eq:pkr} P_{k,R}^n(z)= \sum_{\substack{ k_R+k_B = k \\ k_R \ge 1}} {n_R \choose k_R-1} {n_B \choose k_B} I(k_R, k_B, z) .\end{equation}
We now replace integration in \eqref{eq:I-torus} from $(\T^{d})^{k}$ to $(\R^d)^k$. This is possible provided that $z$ is small enough, so that only the local structure is relevant.
We first notice that, by invariance with respect to translations, we can always fix one variable, say $r_1 = 0$. We thus integrate upon the configurations $\cur{r_i}_{i=2}^{k_R}$, $\cur{b_j}_{j=1}^{k_B}$ such that adding $0$ to the set $\cur{r_i}_{i=2}^{k_R}$ yields a bipartite graph that contains a spanning tree $T$ with edge weights smaller than $z$. 
Now, if $z \le 1/(4k)$, it follows that such tree is contained in a ball of center $0 \in \T^d$ and radius $1/4$, hence it can be isometrically lifted to a tree on $(-1/2,1/2)^d \subseteq \R^d$. Similarly, both $D(\cur{b_j}_{j=1}^{k_B}, z)$, $D(\cur{r_i}_{i=1}^{k_R},z)$  are then contained in a ball of center $0$ and radius $1/2$, hence their volumes computed on $\T^d$ coincide with those of their lift on $(-1/2, 1/2)^d \subseteq \R^d$. To parallel the notation, we therefore the sets $\Theta_{\R^d}(k_R, k_B, z) \subseteq (\R^d)^{k_R} \times (\R^d)^{k_B}$, analogous to $\Theta_{\T^d}(k_R, k_B, z)$ -- notice that 
\[ \Theta_{\R^d}(k_R, k_B, 1) = \Theta(k_R, k_B)\]
defined in \eqref{eq:def-theta} --  and write
\[ D_{\R^d}(A,z) = \cur{x \in \R^d\, : \, \dist(A,x)^p \le z},\]
for $A \subseteq \R^d$ -- notice that $D(A,1) = D(A)$ defined in \eqref{eq:def-da}. For $z \le 1/(4k)$, we have therefore 
\[\begin{split} &  I_R(k_R, k_B, z) =\\
&  \int_{\Theta_{\R^d}(k_R, k_B, z)}  \bra{ 1 - |D_{\R^d}(\cur{b_j}_{j=1}^{k_B}, z)|}^{n_R - k_R}\bra{ 1 - |D_{\R^d} (\cur{r_i}_{i=1}^{k_R}, z)|}^{n_B - k_B} \delta_0(r_1) d r d b,\end{split}\]
where here $drdb$ denotes Lebesgue integration with respect to the remaining $k-1$ variables in $\R^d$. 

Since for every $A \subseteq \R^d$, $z>0$,
\[ D(A, z) = z^{1/p} D( z^{-1/p} A),\]
a change of variable in the integration $r_i z^{-1/p} \mapsto r_i$, $b_i z^{-1/p} \mapsto b_i$ yields
\[\begin{split}  I_R(k_R, k_B, z) =  z^{(k-1)/p} & \int_{\Theta(k_R, k_B)}  \bra{ 1 - z^{d/p} |D(\cur{b_j}_{j=1}^{k_B})|}^{n_R- k_R} \cdot \\
& \quad \cdot \bra{ 1 - z^{d/p} |D(\cur{r_i}_{i=1}^{k_R})|}^{n_B - k_B} \delta_0(r_1) d r d b,\end{split}\]
(notice the exponent $k-1$ instead of $k$ because of the different integration for $r_1$).
We now let $z:=\bra{y/n}^{p/d}$, so that
\[ \begin{split} I_R(k_R, k_B, \bra{y/n}^{p/d}) =  \frac{y^{(k-1)/d}}{n^{k-1}} &  \int_{\Theta(k_R, k_B)}  \bra{ 1 - \frac{y}{n} |D(\cur{b_j}_{j=1}^{n_B})|}^{n_R - k_R} \cdot \\
& \quad \cdot \bra{ 1 - \frac{y}{n} |D(\cur{r_i}_{i=1}^{k_R})|}^{n_B - k_B} \delta_0(r_1) d r d b,\end{split}\]
Since $\Theta(k_R, k_B)$ has finite measure (with respect to $\delta_0(r_1)drdb$), by dominated convergence it follows that
\[\begin{split}
 \lim_{ n \to \infty} &   \int_{\Theta(k_R, k_B)}  \bra{ 1 - \frac{y}{n} |D(\cur{b_j}_{j=1}^{k_B})|}^{n_R - k_R} \bra{ 1 - \frac{y}{n} |D(\cur{r_i}_{i=1}^{k_B})|}^{n_B - k_B} \delta_0(r_1) d r d b\\
& = \int_{\Theta(k_R, k_B)}  \exp\bra{ -y \bra{ \alpha_R |D(\cur{b_j}_{j=1}^{k_B})| + \alpha_B  |D(\cur{r_i}_{i=1}^{k_R})|}} \delta_0(r_1) d r d b \\
& =: \mathcal{I}_R(k_R, k_B, \alpha_R, y).
\end{split}\]
Moreover,
\begin{equation}\label{eq:poisson} \lim_{ n\to \infty} {n_R \choose k_R-1} {n_B \choose k_B} \frac{1}{n^{k-1}} = \frac{ \alpha_R^{k_R-1}}{(k_R-1)!}\frac{ \alpha_B^{k_B}}{k_B!},\end{equation}
so that
\[ \lim_{n \to \infty} P_{k,R}^n( (y/n)^{p/d} ) = \sum_{\substack{ k_R+k_B = k \\ k_R \ge 1}} \frac{ \alpha_R^{k_R-1}}{(k_R-1)!}\frac{ \alpha_B^{k_B}}{k_B!} y^{(k-1)/d} \mathcal{I}_R(k_R, k_B, \alpha_R, y).\]

%
%
%
%
%

To show that convergence is dominated, in view of \eqref{eq:pkr} and the limit \eqref{eq:poisson} it is sufficient to dominate each term 
\[ n^{k-1} I_R(k_R, k_R, (y/n)^{p/d} ),\]
If $A \subseteq \T^d$ is not empty,  then $|D(A, z)| \ge  \omega_d z^{d/p}$ (assuming that $\omega_d z^{d/p} \le 1$). Therefore, if also $k_B \ge 1$, we write
\[ \begin{split}
 I_R(k_R, k_B, z ) & = \int_{\Theta_{\T^d}(k_R, k_B, z)}  ( 1 - |D(\cur{b_j}_{j=1}^{k_B}, z)|)^{n_R - k_R}\bra{ 1 - |D(\cur{r_i}_{i=1}^{k_R}, z)|}^{n_B - k_B} d r d b \\
& \le \int_{\Theta_{\T^d}(k_R, k_B, z)} (1-\omega_d z^{d/p})^{n - k}_+d r d b \\
& \le \exp\bra{ - (n-k) z^{d/p} } | \Theta_{\T^d}(k_R, k_B, z) | \\
& \le k^{k-2} z^{d/p} \exp\bra{ - (n-k) z^{d/p} }.
\end{split}\]
where the bound $| \Theta_{\T^d}(n_R, n_B, z) | \le k^{k-2} z^{d/p}$ follows since every tree (even not necessarily bipartite) with edge weights smaller than $z$ can be iteratively obtained by choosing points $x_{i+1} \in D(\cur{x_j}_{j =1}^i, z)$. 
Substituting $z = (y/n)^{p/d}$ gives the required domination. Notice that, if $k_B = 0$ and $k_R \ge 2$, there is nothing to prove, since $\Theta_{\T^d}(k_R, k_B, z)$ is empty. For the case $k_B=0$, $k_R=1$ we argue similarly obtaining 
\[I_R(1,0, z) \le \exp\bra{ - (n-1) z^{d/p} },\]
which is sufficient to conclude. Arguing similarly for $P_{k,B}^n$ gives
\[ \lim_{n \to \infty} P_{k,B}^n( (y/n)^{p/d} ) = \sum_{\substack{ k_R+k_B = k \\ k_B \ge 1}} \frac{ \alpha_R^{k_R}}{k_R!}\frac{ \alpha_B^{k_B-1}}{(k_B-1)!} y^{(k-1)/d} \mathcal{I}_B(k_R, k_B, \alpha_R, y),\]
with similar definitions. Thus, the validity of assumption \ref{ab-hp-1}) is established.

We notice here that, exchanging the order of integration,
\[ \begin{split} & \int_0^\infty  y^{(k-1)/d} \mathcal{I}_R(k_B, k_R, \alpha_R, y) y^{1/d -1} d y \\
&  = \int_{\Theta(k_R, k_B)} \int_0^\infty  \exp\bra{ -y \bra{ \alpha_R |D(\cur{b_j}_{j=1}^{k_B})| + \alpha_B  D(\cur{r_i}_{i=1}^{k_R})|}} y^{k/d-1} d y \delta_0(r_1) d r d b \\
& = \Gamma(k/d) \int_{\Theta(k_R, k_B)} \bra{ \alpha_R |D(\cur{b_j}_{j=1}^{k_B})| + \alpha_B  | D(\cur{r_i}_{i=1}^{k_R})|}^{-k/d} \delta_0(r_1) d r d b,
\end{split}\]
which, after some manipulations, taking into account also the term with $\mathcal{I}_B$, gives the claimed expression for $\beta_{bMST}(p,d)$.

We next prove that assumption \ref{ab-hp-2}) holds. A lower bound is straightforward, since the maximum weight of the edges is uniformly bounded (by some constant $M= M(d,p) >0$, e.g. $M = d^{p/2}$), it follows that $C_k^n(z)=1$ if $z \ge M$ and $n$ is sufficiently large (recall that we must let first $n \to \infty$ and then $k \to \infty$, so we can assume $n \ge k$).

It follows that
\[
\frac{1}{n^{1-\frac{1}{d}}}\int_{0}^{\infty}\bra{ \EE\sqa{C_{k}^n(z)}-1} dz \geq
-\frac{M}{n^{1-\frac{1}{d}}} \to 0\quad \text{ as $n \to \infty$.}
\]

To obtain an upper bound we use Lemma~\ref{lem:prim} on each $G^n$ -- we can assume $k\ge 2$ and $n \ge k$. Given the sets $C_i$, for $i= 1\ldots, m$ with $m \le n/k$, we choose elements $r_i \in C_i$  such that $b_i:=\n_{G^n}(r_i) \in C_i$ and $\n_{G^n}(b_i) = r_i$. Without loss of generality, we can assume that $r_i \in R^n$, $b_i \in B^n$. We consider the induced subgraph of $G^n_k\subseteq G_k$ obtained by restriction on the nodes $R^n_k := \cur{r_i}_{i=1}^m$, $B^n_k := \cur{b_i}_{i=1}^m$.  With the notation of Lemma~\ref{lem:prim}, we have  
\[ w(i,j) \le \dist_{\T^d}(r_i, b_j)^p,\]
hence
\[\frac{1}{n^{1-\frac{1}{d}}}\int_{0}^{\infty}\bra{ \EE\sqa{C_{k}^n(z)}-1} dz \le \C(G^n_k).\]
We then use Lemma~\ref{lemma:boundbi} (in fact applied on the metric space $\T^d$) to obtain that, for some constant $C = C(p)>0$, 
\[ \C(G^n_k) \le C\bra{ \C(R^n_k) + \sum_{i=1}^m \dist_{\T^d}(r_i, b_i)^p}, \]
where only one summation appears since $\n_{G^n_k} (b_i) = \n_G(b_i) = r_i$. Bounding from above the distance on $\T^d$ with the Euclidean distance, and using Remark \eqref{rem:general-euclidean}, we have the inequality, for some constant $C = C(d,p)>0$,
\[ \C(R^n_k) + \sum_{i=1}^m \dist_{\T^d}(r_i, b_i)^p \le C m^{1-p/d} + \sum_{i=1}^m \min_{j=1,\ldots, n}|r_i - Y_j|^p,\]
where we also used that $b_i = \n_{G^n}(r_i)$. 
We finally apply \eqref{eq:min-moment-p} with $x = r_i$ and the i.i.d. uniform random variables $(Y_j)_{j=1}^n$ to conclude that, again for some further constant $C = C(d,p)>0$, 
\[ \EE\sqa{ \C(R^n_k, B^n_k)} \le C \bra{   m^{1-p/d} + m n^{-p/d}} \le C \bra{ \bra{ \frac{n}{k}}^{1-p/d} + \frac{n}{k} n^{-p/d}}.\]
Dividing by $n^{1-p/d}$ and letting first $n\to \infty$ and then $k \to \infty$ gives the thesis.\qedhere

\end{proof}
\end{theorem}

To transfer the result from the torus $\T^d$ to the cube $[0,1]^d$, we use the fact that points in $[0,1]^d$ can be projected to $\T^d$, and $\dist_{\T^d}(x,y) \le |x-y|$, so that, for any $p>0$, $R$, $B \subseteq [0,1]^d$,
\[ \C^p(R,B |\T^d) \le \C^p(R,B),\]
for any set of points $R$, $B \subseteq [0,1]^d$, where $\C^p(R,B |\T^d)$ denotes the MST cost functional on $\T^d$. Letting $p \to \infty$ yields also $ \C^\infty(R,B |\T^d) \le \C^\infty(R,B)$. A converse inequality is the following one.

\begin{lemma}\label{lem:upper-cube-torus-fat-boundary}
If $\delta \in (0,1/2)$ is such that $\C^\infty(R,B) \le \delta$, then, for every $p>0$,
\[ \C^p(R,B) \le \C^p(R,B |\T^d) +\C^p(R_\delta, B_\delta),\]
where $R_\delta = R\setminus [ \delta, 1-  \delta]^d$, $B_\delta =  B\setminus [ \delta, 1-  \delta]^d$.
\begin{proof}
Indeed, let $T$ be a MST for $R$, $B$ projected on $\T^d$. The assumption gives that for every $\cur{r,b} \in T$, since $|r-b|\le \delta$, it must be $\dist_{\T^d}(r,b) = |r-b|$ if $r \in R\setminus R_\delta$ or  $b \in B \setminus B_\delta$. Therefore, the only obstacle to bound $\C^p(R,B)$ from above by $\C^p(R,B|\T^d)$ is due to edges $\cur{r,b} \in T$ with $r\in R_\delta$ and $b \in B_\delta$, for which $\dist_{\T^d}(r,b)$ may be much smaller than $|r-b|$. However, removing all these edges and adding all the edges of a bipartite Euclidean MST over $R_\delta$, $B_\delta$ yields a connected graph and the desired upper bound. 
\end{proof}
\end{lemma}



We combine the above lemma with the following asymptotic upper bounds.

\begin{lemma}\label{lem:boundexp-bip-any-p}
For $n \ge 1$, let $R^n = \cur{X_i}_{i=1}^{n_R}$, $B^n= \cur{Y_i}_{i=1}^{n_B}$ be (jointly) i.i.d.~uniformly distributed on $[0,1]^d$ with $n_R+n_B = n$ and
\[ \lim_{n \to \infty} \frac{n_R}{n} = \alpha_R \in (0,1), \quad \lim_{n \to \infty} \frac{n_B}{n} = \alpha_B = 1 -\alpha_R.\]
Then, for every $p>0$, there exists a constant $C = C(d,p, \alpha_R)>0$ such that, for $n$ large enough,
\[ \EE\sqa{\C^p\bra{ R^n } }  \le C n^{1-p/d}, \quad \text{and} \quad \EE\sqa{\C^p\bra{ R^n, B^n } }  \le C n^{1-p/d},\]
as well as, for some constant $C=C(d,q, \alpha_R) >0$,
\[ \EE\sqa{  \bra{ \C^\infty\bra{R^n}}^q }\le C\bra{ \frac{\log n}{n}}^{q/d}, \quad \EE\sqa{  \bra{ \C^\infty\bra{R^n, B^n}}^q }\le C\bra{ \frac{\log n}{n}}^{q/d},\]
and finally, for every $a>0$, there exists $C= C(d,a,\alpha_R)>0$, such that for $n$ large enough,
\begin{equation}\label{eq:cinfity-markov} P( \C^\infty\bra{R^n} > C ( \log(n)/n)^{1/d} ) \le \frac{C}{n^a}, \quad  P( \C^\infty\bra{R^n, B^n} > C ( \log(n)/n)^{1/d} ) \le \frac{C}{n^a}.\end{equation}
\begin{proof}
This follows from an application of the space-filling curve technique: consider $\gamma: [0,1] \to [0,1]^d$ such that the push-forward of the uniform measure on $[0,1]^d$ is the uniform measure on $[0,1]^d$ and it is H\"older continuous with exponent $1/d$, i.e., 
\[  C_\gamma := \sup_{s\neq t} \frac{ |\gamma(t)-\gamma(s)|^{1/d}}{|t-s|}  < \infty.\]
While many constructions for $d=2$ are historically well-known, the case of general $d$ is established in detail e.g.\ in  \cite{milne1980peano}. Let then $\bra{Z_i}_{i=1}^{m}$ be i.i.d.\ uniform on $[0,1]$, so that $\bra{\gamma(Z_i)}_{i=1}^{m}$ are i.i.d.\ on $[0,1]^d$. Consider the order statistics
\[ Z_{(1)} = \min_{i=1, \ldots, m} Z_i \le Z_{(2)} \le \ldots \le Z_{(m)} = \max_{i=1, \ldots, m} Z_i\]
and let $T$ be the connected graph on $\cur{\gamma(Z_i)}_{i=1}^{m}$ with edges
\[ \cur{ \cur{\gamma(Z_{(i)}), \gamma(Z_{(i+1)})} \, : \, i=1, \ldots, m-1}.\]
We have
\[ \C^p\bra{\bra{\gamma(Z_i)}_{i=1}^{m}} \le   \sum_{i=1}^{m-1} |\gamma(Z_{(i+1)})- \gamma(Z_{(i)})|^p \le C_\gamma \sum_{i=1}^{m-1} |Z_{(i+1)}- Z_{(i)}|^p.\]
The law of each $Z_{(i+1)} - Z_{(i)}$ is beta $B(1,n)$, so that, for every $q \in \mathbb{N}$,
\[  \EE\sqa{ | Z_{(i+1)} - Z_{(i)} |^q } = \prod_{r=0}^{q-1}\frac{r+1}{m+r+1} \le C(q) m^{-q}.\]
Bounding the $p$-th moment with the $\left\lceil p \right \rceil$-th moment gives  that, for $p>0$, there exists a constant $C(p)$ such that, for every $m$ and $i$,
\[ \EE\sqa{ | Z_{(i+1)} - Z_{(i)} |^p }  \le C(p) m^{-p}\]
The first inequality of the thesis thus follows by summation upon $i= 1\ldots, m-1$ and letting $m=n_R$. For the second inequality, we use Lemma~\ref{lemma:boundbi} and \eqref{eq:min-moment-p}. The remaining inequalities follow analogously, noticing that 
\[ \max_{i=1, \ldots, m-1} | Z_{(i+1)} - Z_{(i)} | = \max_{i=1, \ldots, m} |Z_i - \n(Z_i)| = M^m,\]
with the notation of Proposition~\ref{prop:max-min}.
\end{proof}
\end{lemma}


The following result entails that the expectation of the bipartite Euclidean minimum spanning tree cost on $\T^d$ and on the cube $[0,1]^d$ are much closer than the rate $n^{1-p/d}$.

\begin{proposition}\label{torus-equals-cube}
For $n \ge 1$, let $R^n = \cur{X_i}_{i=1}^{n_R}$, $B^n= \cur{Y_i}_{i=1}^{n_B}$ be (jointly) i.i.d.~uniformly distributed on $[0,1]^d$ with $n_R+n_B = n$ and
\[ \lim_{n \to \infty} \frac{n_R}{n} = \alpha_R \in (0,1), \quad \lim_{n \to \infty} \frac{n_B}{n} = \alpha_B = 1 -\alpha_R.\]
Then, for every $p>0$, there exists a constant $C = C(p,d)>0$ such that, for $n$ large enough, 
\[ \EE\sqa{ \C^p(R^n, B^n)} -\EE[ \C^p(R^n, B^n | \T^d )] \le C n^{1-(p+1)/d} (\log n)^{(p+1)/d}. \]
In particular, for every $p <d$,
\[ \lim_{n \to\infty} \frac{ \EE\sqa{ \C^p(R^n, B^n)}}{n^{1-p/d}} = \beta_{bMST}(p,d)\]
with $\beta_{bMST}(p,d)$ as in Theorem~\ref{thm:main-bipartite-euclidean}. 
\begin{proof}
Let $C>0$ be as in \eqref{eq:cinfity-markov} with $a>p/d$ and let $\delta = C(\log(n)/n)^{1/d}$. Since both $n_R$ and $n_B$ grow linearly with $n \to \infty$, an application of Lemma~\ref{lem:cubes-filled} as in the proof of Proposition~\ref{prop:max-min} ensures that with probability larger than $1-C/n^a$, for every fixed $a>0$, in particular for $a>p/d$ (and a suitable constant $C>0$) all cubes $Q$ with $|Q| = \delta^d$ have non-empty intersections both with $R^n$ and $B^n$.  
Let $E$ the event where both these conditions occur as well as $\C^\infty(R^n, B^n) \le \delta$. 
In the event $E$, we may apply Lemma~\ref{lem:upper-cube-torus-fat-boundary}, otherwise  we simply bound $\C^p(R^n, B^n) \le n d^{p/2}$ and use the fact that $P(E^c)\le C/n^a$ for $a$ arbitrary large.

  We are then reduced to bound from above
\[  \EE\sqa{ \C^p\bra{R^n_\delta, B^n_\delta } I_E }\] 
 where $R^n_\delta  = R^n \setminus [\delta, 1-\delta]^d$, $B^n_\delta  = B^n \setminus [\delta, 1-\delta]^d$. 
 We decompose  $[0,1]^d \setminus [\delta, 1-\delta]^d$ into $\bra{Q_i}_{i=1}^m$, cubes with side length $\delta$,  with $m \le C \delta^{-(d-1)}$, for some $C = C(d)>0$. We consider a bipartite Euclidean minimum spanning tree $T^n_i$ on $R^n_i := R^n \cap Q_i$, $B^n_i :=  B^n \cap  Q_i$ -- both are non empty if the event $E$ occurs -- and then add $m-1$ edges, each with one node in a cube and an adjacent one, to connect all these trees (again this is possible since $E$ occurs). This construction leads to a bipartite spanning tree on $R^n_\delta$, $B^n_\delta$, so that, in $E$,
 \[ \C^p\bra{R^n_\delta, B^n_\delta } \le \sum_{i=1}^m \C^p\bra{ R^n_i, B^n_i }  + m \delta^{p}.\]
 Taking expectation and rescaling from the cube $Q_i$ to $[0,1]^d$, we bound each term in the sum using Lemma~\ref{lem:boundexp-bip-any-p},  writing $N_R(Q_i) = |R^n_i|$, $N_B(Q_i) = |B^n_i|$, that are independent random variables with binomial laws with common parameters $(n, \delta^d)$. It follows that 
 \[ \begin{split} \EE\sqa{ \C^p\bra{ R^n_i, B^n_i } I_E} & \le \delta^p \EE\sqa{ \C^p\bra{ R^n_i, B^n_i } I_{\cur{ N_R(Q_i)>0, N_B(Q_i)>0}}}\\
 & \le C \delta^p \bra{ \EE\sqa{ N_R(Q_i)^{1-p/d}}  + \EE\sqa{ N_B(Q_i)} \EE\sqa{ N_R(Q_i)^{-p/d} I_{\cur{ N_R(Q_i)>0}}}}\\
 & \le C\delta^p \log n,\end{split}\]
 having used that
\[ \EE\sqa{ N_R(Q_i)^{1-p/d}} \le \EE\sqa{ N_R(Q_i)}^{1-p/d} \le C (\log n)^{1-p/d}\]
and
\[ \EE\sqa{ N_R(Q_i)^{-p/d} I_{\cur{ N_R(Q_i)>0}}} \le 1.\]
It follows that, for some (other) constant $C = C(p,d)>0$,
\[\EE\sqa{ \C^p\bra{ R^n_i, B^n_i } I_E} \le C m \delta^p \log n = C \delta^{p-d+1}\log n,\]
that eventually gives the thesis.
\end{proof}
\end{proposition}


We end the proof of Theorem~\ref{thm:main-bipartite-euclidean} with a concentration result to improve from  convergence of expectations to complete convergence.

\begin{proposition}\label{concentration}
Let $d \ge 1$. For $n \ge 1$, let $R^n = \cur{X_i}_{i=1}^{n_R}$, $B^n= \cur{Y_i}_{i=1}^{n_B}$ be (jointly) i.i.d.~uniformly distributed on $[0,1]^d$ with $n_R+n_B = n$ and\[ \lim_{n \to \infty} \frac{n_R}{n} = \alpha_R \in (0,1), \quad \lim_{n \to \infty} \frac{n_B}{n} = \alpha_B = 1 -\alpha_R.\]
Then, for every $p \in (0, d/2)$ if $d \in \cur{1,2}$ or any $p>0$ if $d \ge 3$, complete convergence holds:
\[ \lim_{n \to \infty} \frac{ \C^p(R^n, B^n)- \EE\sqa{ \C^p(R^n, B^n)}}{n^{1-p/d}} = 0.\]

\begin{proof}

Consider the function
\begin{equation*}\label{eq:function-concentration}
\bra{[0,1]^{d}}^{n_R+n_B} \ni  (x,y) = (x_1, \ldots, x_{n_R}, y_1, \ldots, y_{n_B}) \mapsto f(x,y)=  \C^p\bra{\cur{x_i}_{i=1}^{n_R},  \cur{y_j }_{j=1}^{n_B}}.
\end{equation*}
We argue separately for $p< 1$ and $p\ge 1$. In the former case, if $(x,y)$, $(x',y') \in \bra{[0,1]^{d}}^{n_R+n_B}$ differ only on a single coordinate, say for simplicity $x_1 \neq x_1'$, then, letting $T$ denote the Euclidean bipartite minimum spanning tree on $\cur{x_i}_{i=1}^{n_R}$, $\cur{y_j}_{j=1}^{n_B}$,
\[ \begin{split} f(x',y') & \le f(x,y) + \sum_{\cur{y_j, x_1} \in T} | y_j-x_1'|^p - | y_j-x_1|^p\\
& \le f(x,y)  + \sum_{\cur{y_j, x_1} \in T} |x_1-x_1'|^p  \quad \text{(by the triangle inequality, since $p\le 1$)}\\
& \le f(x,y) + \deg_{T}(x_1)  |x_1-x'_1|^p \\
& \le f(x,y) + \Delta(T) d^{p/2},
\end{split}\]
where $\Delta(T)$ denote the maximum degree of $T$. Arguing symmetrically, we obtain
\[ |f(x,y) - f(x',y')| \le  \Delta(T) d^{p/2} + \Delta(T') d^{p/2}.\]
By Lemma~\ref{lem:mcdiarmid-sobolev} with $E_i = [0,1]^d$ and $g_i(x,y) = \Delta(T) d^{p/2}$ (for every $i=1,\ldots, n$)  we obtain, for every $q\ge 2$, that there exists $C = C(q)>0$ such that
\[\begin{split} \EE\sqa{ \abs{ \C^p\bra{R^n, B^n} - \EE\sqa{ \C^p\bra{R^n, B^n}}}^q } & \le C n^{q/2} \EE\sqa{ \bra{ \Delta(T^n) d^{p/2}}^q}\\
& \le  C' n^{q/2} (\log n)^q, 
\end{split}\]
where we used \eqref{eq:moment-delta-tn} and $C'= C(d,q,\alpha_R, \alpha_b)>0$ is a constant. Dividing both sides by $n^{q(1-p/d)}$ and using Markov inequality yields, for every $\epsilon>0$,
\[ P\bra{ \abs{ \C^p\bra{R^n, B^n} - \EE\sqa{ \C^p\bra{R^n, B^n}}}/n^{1-p/d}>\epsilon } \le C' \epsilon^{-q} n^{q(p/d-1/2)} (\log n)^q,\]
that is summable if $p/d-1/2<0$, i.e., $p<d/2$, and $q$ is sufficiently large.

In the case $p \ge 1$, we use the fact that $f$ is Lipschitz (being minimum of Lipschitz functions) with a.e.\ derivative given by 
\[ \nabla_{x_i} f(x,y) =  \sum_{\substack{ j=1,\ldots, n_B \\ \cur{x_i, y_j} \in T}} p |x_i-y_j|^{p-2} (x_i-y_j),\]
\[ \nabla_{y_j} f(x,y) =  \sum_{\substack{ i=1,\ldots, n_R \\ \cur{x_i, y_j} \in T}} p |x_i-y_j|^{p-2} (x_i-y_j),\]
We bound from above, using Cauchy-Schwartz inequality
\[  |\nabla_{x_i} f(x,y)|^2  \le p^2  \Delta(T) \sum_{\substack{ j=1,\ldots, n_B \\ \cur{x_i, y_j} \in T}} |x_i-y_j|^{2(p-1)}.\] It follows that the (Euclidean) norm of the derivative is bounded above by
\[ |\nabla f(x,y)|^2 \le 2 p^2 \Delta(T)  \C^{2(p-1)}(x,y) \le 2 p^2 \Delta(T) \bra{ \C^\infty(x,y)}^{2(p-1)} n,\]
where we used the fact that the minimum spanning tree $T$ does not depend on the choice of $p$, see Remark~\ref{rem:p-dependence-mst}. For every $q>0$ Lemma~\ref{lem:lower-bound-max-deg} and Lemma~\ref{lem:boundexp-bip-any-p} yield that, for some constant $C = C(d,p,q)>0$ and $n$ sufficiently large,
\[\begin{split} \EE\sqa{ \Delta(T^n)^{q/2} \bra{ \C^\infty(R^n,B^n)}^{q(p-1)}} & \le  \EE\sqa{ \Delta(T^n)^{q}}^{1/2} \EE\sqa{\bra{ \C^\infty(R^n,B^n)}^{2q(p-1)}}^{1/2}\\
& \le  C \log(n)^{q/2} \bra{ \frac{ \log n}{n}}^{q(p-1)/d}
\end{split}\]
It follows that (possibly for a larger constant $C$)
\[ \EE\sqa{ |\nabla f( \cur{X_i}_{i=1}^n, \cur{Y_i}_{i=1}) |^q } \le C \log(n)^{q\bra{ 1/2+(p-1)/d}} n^{q(1/2- (p-1)/d)}.\]
Poincaré inequality for the uniform measure on the unit cube, see e.g.\ the argument in \cite[Prop. 2.8]{ledoux2001concentration}, gives that, for some constant $C = C(q)>0$,
\[ \EE\sqa{\abs{ f(  \cur{X_i}_{i=1}^n, \cur{Y_i}_{i=1})  -\EE\sqa{f(  \cur{X_i}_{i=1}^n, \cur{Y_i}_{i=1})}}^q  } \le C(q) \EE\sqa{ |\nabla f( \cur{X_i}_{i=1}^n, \cur{Y_i}_{i=1}) |^q },\]
thus using Markov inequality, for every $\epsilon>0$, we have
\[  P\bra{ \abs{ \C^p\bra{R^n, B^n} -  \EE\sqa{ \C^p\bra{R^n, B^n}}}/n^{1-p/d}>\epsilon }\le C  \frac{ ( \log n)^{q\bra{ 1/2+(p-1)/d}}}{\epsilon^q  n^{q(1/2-1/d)}},\]
which is summable if $q$ is large enough and $d \ge 3$. 
\end{proof}
\end{proposition}

\appendix

\section{A concentration inequality in $L^q$.}\label{app:mcdiarmid}

McDiarmid inequality \cite{mcdiarmid1989method} is a simple but effective concentration inequality often used in random combinatorial optimization problems.  An interpretation is that the oscillations of a function of many independent random variables are bounded by its Lipschitz norm (with respect to a Hamming-type distance). The usual proof relies concentration inequalities for discrete time exponential martingales. In this section we show an analogous result where we replace the Lipschitz condition with a ``Sobolev'' one and use Burkholder-Gundy inequalities instead.

\begin{lemma}\label{lem:mcdiarmid-sobolev}
Let $\bra{(E_i, \mathcal{E}_i)}_{i=1}^n$ be measurable spaces, set $E = \prod_{i=1}^n E_i$ and let
\[ f: E \to \R, \quad g_i: E \to [0,\infty] \quad \text{for $i=1,\ldots, n$,}\]
be such that, for every $i \in \cur{1,\ldots, n}$, for every $x$, $x'\in E$ with
\[ x= (x_1,\ldots, x_i, \ldots, x_n) \quad  x'= (x_1, \ldots, x'_i, \ldots, x_n)\]
then
\[  \abs{ f(x) - f(x') } \le g_i(x) +g_i(x').\]
For every $q \ge 2$, there exists $C= C(q)>0$ such that, if $X = (X_i)_{i=1}^n$ are independent random variables, with $X_i: \Omega \to E_i$, then 
\[ \EE\sqa{  | f(X) - \EE\sqa{f(X)}|^q } \le C(q) n^{q/2-1} \sum_{i=1}^n \EE\sqa{ g_i(X)^q}.\]
\begin{proof}
Write $\EE_{i}$ for the conditional expectation with respect to the variables $\bra{X_j}_{j\le i}$. In particular, $\EE_{0} = \EE$ and $\EE_{n}$ is the identity operator. We write $f(X)- \EE\sqa{f(X)}$ as a sum of martingale differences
 \[ f(X) - \EE\sqa{f(X)} = \sum_{i=1}^n \EE_{i} \sqa{f(X)}-\EE_{{i-1}} \sqa{f(X)}.\]
Burkholder-Gundy inequality \cite{burkholder_extrapolation_1970} gives, for some constant $C=C(q) \ge 0$,
\begin{equation}\label{eq:burkholder}  \EE\sqa{  | f(X) - \EE\sqa{f(X)}|^q } \le C(q) \EE\sqa{ \bra{ \sum_{i=1}^n \abs{ \EE_{ i} \sqa{f(X)}-\EE_{ {i-1}} \sqa{f(X)}}^2 }^{q/2}}.\end{equation}
To simplify notation, we introduce a copy of the independent variables $(X_i')_{i=1}^n$ defined on a different space $\Omega'$, and write $\EE'$ for expectation with respect to such variables, so that, because of independence, for every $i=0,\ldots, n$,
\[  \EE_{ i} \sqa{f(X)} = \EE'\sqa{ f(X_1, \ldots, X_i, X'_{i+1}, \ldots X'_n)}.\]%
\[ \begin{split} & \abs{  \EE_{ i} \sqa{f(X)} - \EE_{ i-1} \sqa{f(X)}}\\
&  \quad  = \abs{ \EE'\sqa{ f(X_1, \ldots, X_i, X'_{i+1}, \ldots X'_n)} - \EE'\sqa{ f(X_1, \ldots, X_{i-1}, X'_i,  \ldots X'_n)}} \\
& \quad \le \EE'\sqa{ \abs{ f(X_1, \ldots, X_i, X'_{i+1}, \ldots X'_n) -  f(X_1, \ldots, X_{i-1}, X'_i,  \ldots X'_n)}}\\
& \quad \le \EE'\sqa{ g_i (X_1,\ldots, X_i, X'_{i+1}, \ldots, X'_n) + g_i (X_1,\ldots, X_{i-1}, X'_i,  \ldots, X'_n)}\\
& \quad = \EE_{ i} \sqa{g_i} +\EE_{ {i-1}} \sqa{g_i}. 
\end{split}
\]
Using these inequalities in \eqref{eq:burkholder} yields
\[ \begin{split}  \EE\sqa{  | f(X) - \EE\sqa{f(X)}|^q }  & \le C(q) \EE\sqa{ \bra{ \sum_{i=1}^n  \bra{\EE_{ i} \sqa{g_i} +\EE_{ {i-1}} \sqa{g_i} }^2 }^{q/2}}\\
& \le C(q) n^{q/2-1} \sum_{i=1}^n \EE\sqa{ \bra{\EE_{ i} \sqa{g_i} +\EE_{ {i-1}} \sqa{g_i} }^q}\\
& \le \tilde{C}(q)  n^{q/2-1} \sum_{i=1}^n \EE\sqa{ \EE_{ i} \sqa{g_i}^q+\EE_{ {i-1}} \sqa{g_i}^q }\\
& \le 2 \tilde{C}(q) n^{q/2-1}  \sum_{i=1}^n\EE\sqa{ g_i^q},
\end{split}\]
hence the thesis.
\end{proof}
\end{lemma}


\printbibliography

 \end{document}